\documentclass[12pt]{amsart}
\usepackage{amsfonts,amsmath,amsthm,amssymb,}
\usepackage[usenames,dvipsnames,svgnames,table]{xcolor}

\usepackage[margin=1.5in]{geometry} 

\usepackage{todonotes}

\numberwithin{equation}{section}

\newtheorem{thm}{Theorem}[section]
\newtheorem*{thm*}{Theorem}
\newtheorem{prop}[thm]{Proposition}
\newtheorem*{prop*}{Proposition}
\newtheorem{question}[thm]{Question}

\newtheorem{cor}[thm]{Corollary}
\newtheorem{lemma}[thm]{Lemma}
\newtheorem{key}[thm]{Key Lemma}

\theoremstyle{definition}
\newtheorem{defin}[thm]{Definition}
\newtheorem{nota}[thm]{Notation}
\newtheorem{notarem}[thm]{Notation and Remarks}

\newtheorem{example}[thm]{Example}
\newtheorem{remark}[thm]{Remark}
\newtheorem*{remark*}{Remark}
\newtheorem*{remarks*}{Remarks}
\newtheorem{hypoth}[thm]{Simplifying hypotheses}
\newtheorem{ab}[thm]{}
\newtheorem{notconv}[thm]{Notational Convention}

\newcommand{\Aut}{\operatorname{Aut}}
\newcommand{\mytrace}{\operatorname{trace}}
\newcommand{\Ric}{\operatorname{Ric}}
\newcommand{\nilrad}{\operatorname{Nilrad}}
\newcommand{\Der}{\operatorname{Der}}
\newcommand{\ad}{\operatorname{ad}}
\newcommand{\Ad}{\operatorname{Ad}}

\newcommand{\Rc}{\operatorname{Rc}}
\newcommand{\Id}{\operatorname{Id}}
\newcommand{\Real}{\operatorname{Real}}
\newcommand{\diff}{\operatorname{Diff}}
\newcommand{\abc}{\operatorname{ab}}
\newcommand{\la}{\langle}

\newcommand{\ra}{\rangle}

\newcommand{\Ip}{\la\,,\,\ra}
\newcommand{\s}{\mathfrak{s}}
\newcommand{\g}{\mathfrak{g}}
\newcommand{\n}{\mathfrak{n}}
\newcommand{\rf}{\mathfrak{r}}
\newcommand{\lf}{\mathfrak{l}}

\newcommand{\mf}{\mathfrak{m}}
\newcommand{\h}{\mathfrak{h}}
\newcommand{\q}{\mathfrak{q}}
\newcommand{\kf}{\mathfrak{k}}
\newcommand{\p}{\mathfrak{p}}
\newcommand{\m}{\mathfrak{m}}
\newcommand{\uf}{\mathfrak{u}}

\newcommand{\af}{\mathfrak{a}}
\newcommand{\cf}{\mathfrak{c}}
\newcommand{\bfrak}{\mathfrak{b}}
\newcommand{\gl}{\mathfrak{gl}}
\newcommand{\R}{\mathbf{R}}
\newcommand{\Z}{\mathbf{Z}}
\newcommand{\C}{\mathbf{C}}

\newcommand{\Gt}{\hat G}
\newcommand{\Lt}{\hat L}
\newcommand{\Ht}{\hat H}

\newcommand{\Mcal}{{\mathcal M}}
\newcommand{\Rcal}{{\mathcal R}}
\newcommand{\Scal}{{\mathcal S}}

\newcommand{\dercntwo}{\Der(\n_2)^{\mathfrak{a}_2 + \rho(\mathfrak s_1)}}

\newcommand{\Isom}{{\operatorname{Isom}}}

\begin{document}

\title[Einstein solvmanifolds have maximal symmetry]{Einstein solvmanifolds have maximal symmetry}
\author[C.S. Gordon]{Carolyn S. Gordon$^1$}
\address{Department of Mathematics, Dartmouth College, Hanover, NH 03755-1808, USA}
\email{carolyn.s.gordon@dartmouth.edu}
\thanks{$^1$Research partially supported by National Science Foundation
grant DMS-0906168}
\author[M.R. Jablonski]{Michael R. Jablonski$^2$}
\address{Department of Mathematics, University of Oklahoma, Norman, OK 73019-3103, USA}
\email{mjablonski@math.ou.edu}
\thanks{$^2$Research partially supported by National Science Foundation
grant DMS-1105647}

\maketitle

\makeatletter
\providecommand\@dotsep{5}
\makeatother

%
%

%

\begin{abstract}  All known examples of homogeneous Einstein metrics of negative Ricci curvature can be realized as left-invariant Riemannian metrics on solvable Lie groups.   After defining a notion of maximal symmetry among left-invariant Riemannian metrics on a Lie group, we prove that any left-invariant Einstein metric of negative Ricci curvature on a solvable Lie group is maximally symmetric.   This theorem is motivated  both by the Alekseevskii Conjecture and by the question of stability of Einstein metrics under the Ricci flow.   We also address questions of existence of maximally symmetric left-invariant Riemannian metrics more generally.
\end{abstract}

Einstein metrics have long held a distinguished place among Riemmanian metrics.  Even among homogeneous spaces, however, there is no classification of Einstein spaces despite decades of study.   (An exception is the case of homogeneous spaces of Ricci curvature zero; such homogeneous spaces are not only Ricci flat but actually flat.)
In the case of homogeneous Einstein spaces of positive Ricci curvature, there are so many examples that no classification is expected to exist.  For example, every compact semisimple Lie group of dimension $6$ or greater admits at least two such Einstein metrics, and many admit more than this.
However, in the setting of homogeneous Einstein spaces with negative Ricci curvature, there is considerably more structural rigidity; see, for example,\cite{Heber, LauretLafuente:StructureOfHomogeneousRicciSolitonsAndTheAlekseevskiiConjecture,
JP:TowardsTheAlekseevskiiConjecture}, and there is a reasonable hope for an eventual classfication.   
Presently, all known examples of homogeneous Einstein spaces of negative Ricci curvature are isometric to left-invariant metrics on solvable Lie groups, and there is mounting evidence that Einstein metrics on solvable Lie groups exhaust the class of homogeneous Einstein spaces with negative Ricci curvature.   

In this article, we study left-invariant Einstein metrics with negative Ricci curvature on solvable Lie groups and show that they satisfy a new, rather strong, geometric property:

\begin{thm}[\bf{Main Theorem}]\label{mt} If a solvable Lie group $S$ admits a left-invariant Einstein metric $g$ of negative Ricci curvature, then $g$ is maximally symmetric.
\end{thm}

The notion of maximal symmetry is defined in Definition~\ref{def.maxsym} below.  

The Main Theorem is motivated both by an interest in understanding maximally symmetric Riemannian metrics and by questions concerning homogeneous Einstein metrics and Ricci flow.  We will address each of these motivations in this introductory section.

\subsection*{Maximally symmetric Riemannian metrics}
Given the implications of geometry on the topology of a Riemannian manifold, it is natural to ask whether a given manifold admits a distinguished metric and to address the geometry of any such metric.   
As a first attempt at seeking a distinguished metric, let's say that a Riemannian metric $g$ on a manifold $M$ has \emph{maximal symmetry} if for any other Riemannian metric $h$ on $M$ we have
	$$ \Isom(M,h) \subset \Isom(M,\psi^*g),  $$
for some diffeomorphism $\psi \in \diff(M)$. Here $ \Isom(M,h)$ is the full isometry group of $(M,h)$.   With this notion, the maximally symmetric metrics on the 2-sphere are, as expected, precisely the constant curvature metrics.    This fact may be deduced from the works of H.~Poincar\'e and L.E.J.~Brouwer; however, we provide a brief justification here using modern techniques.  Applying the (normalized) Ricci flow to any metric on $S^2$, one has that the Ricci flow converges to a metric of constant curvature \cite{Chow:TheRicciFlowOnThe2Sphere}.  Furthermore, the Ricci flow preserves isometries  \cite{Chen-Zhu:UniquenessRicciFlowCompleteNoncompact} and so the isometry group of the initial metric acts isometrically on the limit metric.  As the round metric is unique up to scaling and diffeomorphism, we conclude that it has maximal symmetry.

 Interestingly, this result does not hold for all spheres as there exist finite groups acting on the sphere $S^n$ which cannot be realized as subgroups of $O(n+1)$. Thus, higher-dimensional spheres do not admit maximally symmetry metrics in this sense.  This failure suggests that the notion of maximal symmetry above is too demanding; perhaps one needs to restrict attention to a subclass of all the Riemannian metrics, for example.

In the setting of Lie groups, we define a notion of maximal symmetry among left-invariant metrics:  

\begin{defin}\label{def.maxsym} A left-invariant Riemannian metric $g$ on a Lie group $G$ will be said to be \emph{maximally symmetric} if for any other left-invariant Riemannian metric $h$ on $M$ we have
	$$ \Isom(M,h) \subset \Isom(M,\psi^*g),  $$
for some $\psi\in \Aut(G)$.
\end{defin}

This definition appears to be more robust.  While it is still the case that not every Lie group admits a maximally symmetric left-invariant Riemannian metric, large classes of Lie groups do.   Two questions naturally arise:   (i) Which Lie groups admit maximally symmetric left-invariant metrics?  (ii) Are left-invariant metrics that are geometrically distinguished --e.g., are those with special curvature properties--maximally symmetric?

We briefly address question (i) in Section 1.   The primary focus of this paper is the Main Theorem above, which is an instance of question (ii).



\subsection*{Einstein motivations}  The Main Theorem is partially motivated by the well-known Alekseevskii Conjecture and stability questions for the Ricci flow around Einstein metrics.

The Alekseekskii Conjecture asserts that  every connected homogeneous Einstein manifold of negative Ricci curvature is diffeomorphic to $\R^n$.  An only slightly stronger version of the conjecture is that every such Einstein manifold is a solvmanifold, i.e., it is isometric to a solvable Lie group with a left-invariant Riemannian metric.  A major hurdle to proving the conjecture is to show that if a homogeneous space $G/K$ of a semisimple Lie group $G$ admits a left-invariant Einstein metric of negative Ricci curvature, then $K$ must be a maximal compact subgroup of $G$ (in which case $G/K$ is a symmetric space and is isometric to a solvmanifold, since an Iwasawa subgroup of $G$ acts simply transitively).  Recently, P. Petersen and the second author  \cite[Corollary 1.3]{JP:TowardsTheAlekseevskiiConjecture} showed that if $G$ is semisimple of noncompact type and $G/K$ admits a left-invariant Einstein metric $g$, then the identity component $\Isom_0(G,g)$ of the full isometry group consists only of $G$ itself.  Thus any counterexample to the Alekseevskii Conjecture of this form would in a sense be \emph{minimally} symmetric.  In particular, it would have minimal possible isotropy algebra among all possible left-invariant metrics on $G/K$.   In contrast, the Main Theorem asserts that Einstein metrics of negative Ricci curvature on solvable Lie groups have maximal possible isotropy.

A second motivation for the Main Theorem above comes from studying the Ricci flow.  It is an open question whether Einstein metrics on solvable Lie groups are stable under the Ricci flow \cite{Arroyo:TheRicciFlowInAClassOfSolvmanifolds,
JPW:LinearStabilityOfAlgebraicRicciSolitons,
WilliamsWu:DynamicalStabilityOfAlgebraicRicciSolitons}.  If such spaces were stable, then one would be able to deduce that (locally) their isometry groups are maximal.

\subsection*{Outline of the proof of Main Theorem
~\ref{mt}.}  We begin with an arbitrary left-invariant metric $h$ on $S$ and let $G$ be the full isometry group and $L$ the isotropy subgroup.   The theorem is equivalent to the statement that $G/L$ admits a $G$-invariant Einstein metric, which is in turn equivalent to the condition that \emph{some} simply transitive solvable subgroup of $G$ admits a left-invariant Einstein metric whose isometry group includes $G$.   

At the outset, we replace the original solvable group by one (that we will continue to denote by $S$ in this introductory 
outline) in so-called ``standard position'' in $G$.   Using results of \cite{Heber,LauretStandard} concerning the structure of Einstein solvmanifolds along with results of \cite{GordonWilson:IsomGrpsOfRiemSolv} concerning isometry groups of solvmanifolds, we show that $S$ can be written as a semi-direct product $S=S_1\ltimes S_2$,  where $S_1$ is an Iwasawa subgroup of a semisimple Levi factor of $G$ and $S_2=S\cap \operatorname{Rad}(G)$.  Moreover, this new solvable group $S$ admits an Einstein metric isometric to the one on the original solvable group.

The key step in proving that this Einstein metric is $G$-invariant is to prove the following lemma, which is perhaps of interest in its own right:

\begin{lemma}\label{lem} Suppose that $S$ is a semi-direct product $S=S_1\ltimes S_2$ of solvable Lie groups satisfying the following hypotheses:
\begin{itemize} 
\item $S_1$ isomorphically embeds as an Iwasawa subgroup in a semisimple Lie group $G_1$.   
\item The adjoint action of $S_1$ on the Lie algebra $\operatorname{Lie}(S_2) $ extends to a representation of $G_1$ on $\operatorname{Lie}(S_2)$.
\end{itemize}
Then, $S$ admits a left-invariant Einstein metric of negative Ricci curvature if and only if $S_2$ does.  In this case, the Einstein metric $g$ on $S$ may be chosen so that its restriction to $S_2$ is Einstein, its restriction to $S_1$ is symmetric, and the Lie algebras of $S_1$ and $S_2$ are orthogonal.
\end{lemma}

The ``if'' statement and the final statement are elementary.  On the other hand, the crucial  ``only if'' statement exploits the deep relationship between left-invariant Einstein metrics of negative Ricci curvature on solvable Lie groups and geometric invariant theory.  This relationship first appeared in the work of Heber \cite{Heber} and was subsequently refined by Lauret \cite{LauretStandard} and Nikolayevsky \cite{Nikolayevsky:EinsteinSolvmanifoldsandPreEinsteinDerivation}.  The connection with geometric invariant theory arises as follows:  The question of existence of an Einstein metric on a solvable group $S$ reduces to the question of existence of a nilsoliton metric on the nilradical $N$ of $S$.  Denoting the Lie bracket of Lie$(N)$ by $\mu$, we may identify Lie$(N)$ with $\{ \mathbb R^n , \mu \}$  and view $\mu$ as an element of $V = \wedge (\mathbb R^n)^* \otimes \mathbb R^n$, the space of brackets on $\mathbb R^n$.  The group $GL_n\mathbb R$ acts on $V$ and there exists a naturally defined subgroup $G_\phi \subset GL_n\mathbb R$ with the following property:
	$$\mbox{The orbit } G_\phi \cdot \mu \subset V \mbox{ is closed if and only if } N \mbox{ admits a nilsoliton metric.}     $$
In the notation of Lemma~\ref{lem}, the nilradical $N_2$ of $S_2$ is a normal subgroup of the nilradical $N$ of $S$.  Denote by $\mu_2$ the Lie bracket of Lie$(N_2) \subset \operatorname{Lie}(N)$, and denote the associated group by $G_{\phi_2}$.  To prove the ``only if'' statement, assume that $S$ admits a left-invariant Einstein metric of negative Ricci curvature.  We show that $G_{\phi_2} \cdot \mu_2$ inherits the topological property of being closed from the analogous property of the orbit $G_\phi\cdot\mu$.  In this way, $N_2$ obtains a nilsoliton metric and then $S_2$ obtains an Einstein metric.

The completion of the proof of the Main Theorem, given the Lemma, proceeds as follows:  Since we have already established the existence of an Einstein metric on $S$, the forward direction of the lemma gives us an Einstein metric on $S_2$ and a nilsoliton metric on the nilradical $N_2$ of $S_2$.  Using the fact that nilsoliton metrics on nilpotent Lie groups are maximally symmetric (see Section \ref{sec: state of knowledge on max symm on Lie groups}) and that $S_2$ and $N_2$ are normal in $G$, we find that some nilsoliton metric on $N_2$ -- and the resulting Einstein metric on $S_2$ -- are $\Ad(L)$-invariant and further $\Ad(G_1)$ acts by self-adjoint transformations on Lie$(N_2)$.  It is then straightforward to show that the extension of this metric on $S_2$ to an Einstein metric on $S$, given in the easier direction of the Lemma, is $G$-invariant.  (This Einstein metric on $S$ may differ from the original one by an automorphism).

\subsection*{Organization}  The paper is organized as follows:   In Section \ref{sec: state of knowledge on max symm on Lie groups} we address maximal symmetry metrics on Lie groups.  The structure theory of Einstein solvmanifolds and their automorphism groups is reviewed in Section~\ref{einstauts}.   Section~\ref{main} contains   the proof of the Main Theorem~\ref{mt} modulo Lemma~\ref{lem}.   The proof of the lemma is presented in Section \ref{sec: proof of key lemma} after first addressing the prerequisite Geometric Invariant Theory in Section \ref{sec: technical lemmas on GIT}.   The question of Einstein extensions is further addressed in Section~\ref{exts}.

\section{Existence of maximally symmetric left-invariant metrics}\label{sec: state of knowledge on max symm on Lie groups}

In this section we address the question of which Lie groups admit maximally symmetric left-invariant Riemannian metrics, as defined in Definition~\ref{def.maxsym}.

\begin{prop}\label{prop.normal} Suppose $G$ satisfies:
\begin{enumerate}
\item $\Aut(G)$ has only finitely many components.   (This condition is always satisfied if $G$ is simply-connected.)
\item For every left-invariant metric $h$ on $G$, $\Isom(G,h)< G\rtimes \Aut(G)$, where $\Isom(G,h)$ is the full isometry group of $h$.
\end{enumerate}
Then $G$ admits maximally symmetric left-invariant Riemannian metrics.
\end{prop}

The second hypothesis is equivalent to the condition that $G$  (more precisely, the group of left-translations of $G$) is a normal subgroup of $\Isom(G,h)$ for every choice of $h$.

\begin{proof}

  Let $K<\Aut(G)$ be a maximal compact subgroup of $\Aut(G)$, and let $g$ be a $K$-invariant, left-invariant Riemannian metric.   For $h$ any left-invariant metric, the first hypothesis implies that $\Isom(G,h)= G\rtimes L$ for some compact subgroup $L<\Aut(G)$.   The first hypothesis guarantees that all maximal compact subgroups of $\Aut(G)$ are conjugate.  Thus there exists $\tau\in\Aut(G)$ such that $\tau L\tau^{-1}<K$, and we have $\Isom(G,h)<\Isom(G,\tau^*g)$.
\end{proof}

\begin{cor}\label{cor.pos} Let $G$ be a connected Lie group satisfying any one of the following conditions:
\begin{itemize}
\item $G$ is compact and simple.
\item $G$ is semisimple of noncompact type.  
\item $G$ is a simply-connected, completely solvable unimodular Lie group.  (E.g., all simply-connected nilpotent Lie groups satisfy this condition.)

\end{itemize}
Then $G$ admits maximally symmetric left-invariant Riemannian metrics.
\end{cor}

\begin{proof}  We apply Proposition~\ref{prop.normal}.   The first hypothesis of the proposition is trivially satisfied in all three cases.    The fact that each of these types of Lie groups satisfy the second hypothesis is proven in \cite{Ochiai-Takahashi},  \cite{G79}, and \cite{GordonWilson:IsomGrpsOfRiemSolv}, respectively.   (See also \cite{Wilson} for the special case of nilpotent Lie groups.)
\end{proof}

In some cases included in Corollary~\ref{cor.pos}, one can identify a maximally symmetric left-invariant metric.  For compact simple Lie groups, the maximally symmetric left-invariant metrics are precisely the bi-invariant metrics.   For semisimple Lie groups of non-compact type, the first author has shown that the natural metric coming from the Killing form is maximally symmetric, see \cite{G79}.  The second author proved the following:

\begin{prop}\cite{Jablo:ConceringExistenceOfEinstein}\label{completely_solv} If a completely solvable unimodular Lie group admits a left-invariant Ricci soliton metric $g$, then $g$ is maximally symmetric.
\end{prop}

While Corollary~\ref{cor.pos} gives large families of Lie groups that admit maximally symmetric left-invariant Riemannian metrics, the existence of such metrics is far from universal.   

\begin{prop}\label{cmpt} There exist compact semisimple Lie groups $G$ that do not admit a maximally symmetric left-invariant Riemannian metric. 
\end{prop}

\begin{proof} 
Suppose that $g_0$ is a maximally symmetric left-invariant Riemannian metric on $G$.   We first show that $g_0$ must be bi-invariant.  Let $g$ be a bi-invariant metric on $G$.  Then there exists $\tau\in\Aut(G)$ such that 
$$G\rtimes\Aut(G)=\Isom(G,g) < \Isom(G,\tau^*g_0).$$
In particular, $\tau^*g_0$ is invariant under all inner automorphisms and thus under right, as well as left translations.  I.e., $\tau^*g_0$ is bi-invariant.   But then $\tau^*g_0=g_0$, so $g_0$ is bi-invariant.

Definition~\ref{def.maxsym}, together with the fact that $g_0$ is invariant under $\Aut(G)$, thus implies that for every left-invariant metric $h$, we have $\Isom(G,h)<\Isom(G,g_0)=G\rtimes \Aut(G)$ and hence $G$ is normal in $\Isom(G,h)$.  However, Ozeki \cite{Ozeki} proved that there exist left-invariant Riemannian metrics $h$ on some compact semisimple Lie groups $G$ such that the group of left-translations of $G$ is not a normal subgroup of $\Isom(G,h)$.

\end{proof}

\begin{remark}  In the same article \cite{Ozeki} cited in the proof of Proposition~\ref{cmpt}, Ozeki showed that for every left-invariant metric $h$ on a compact semisimple Lie group $G$, there exists a normal subgroup $G'$ of $\Isom(G,h)$ that is isomorphic to $G$.  It is thus easy to see that bi-invariant metrics $g$ on $G$ satisfy the following weaker notion of maximal symmetry:  For every left-invariant Riemannian metric $h$ on $G$, the isometry group of $h$ is isomorphic to a subgroup of $\Isom(G,g)$.  
\end{remark}

\begin{example}\label{ex.sl} Let $S$ be the connected, simply-connected solvable Lie group given by the direct product $S=S_1\times \R$ where $S_1$ is the Iwasawa subgroup of $SL(2,\R)$, i.e., $S_1$ is the unique connected, simply-connected, non-abelian, two-dimensional solvable Lie group.  We show that $S$ cannot admit a maximally symmetric left-invariant Riemannian metric.  First note that for any left-invariant metric on a three-dimensional Lie group, the full isometry group must have dimension 3, 4 or 6, since the isotropy algebra must be isomorphic to a subalgebra of $\mathfrak{so}(3)$.  Moreover, every three-dimensional manifold with a six-dimensional isometry group must have constant sectional curvature.   Since $S$ does not admit a left-invariant metric of constant curvature, the isometry group of any left-invariant metric on $S$ must have dimension at most four.  We will exhibit a pair of left-invariant metrics $h_1$ and $h_2$ on $S$ such that the identify components $\Isom_0(S,h_1)$ and $\Isom_0(S,h_2)$ are non-isomorphic four-dimensional Lie groups.   If $S$ admitted a maximally symmetric left-invariant Riemannian metric $g$, then $\Isom(S,g)$ would have to contain isomorphic copies of both $\Isom_0(S,h_1)$ and $\Isom_0(S,h_2)$.   This is impossible since $\Isom(S,g)$ can itself be at most four-dimensional.

We construct the two metrics.   
We define $h_1$ to be the direct product of the hyperbolic metric on $S_1$ and a Euclidean metric on $\R$.   Then $$\Isom_0(S,h_1)= PSL(2,\R)\times \R.$$   To define $h_2$, first consider a left-invariant metric $h$ on the universal cover $\widetilde{SL}(2,\R)$ of $SL(2,\R)$, defined by an $\Ad(K)$-invariant inner product on the Lie algebra $\mathfrak{sl}(2,\R)$, where $K=SO(2,\R)$.   The identity component of the isometry group of $h$ is given by 
$$\Isom_0(\widetilde{SL}(2,\R), h)=(\widetilde{SL}(2,\R)\times \tilde{K})/D.$$
(See \cite{G79}.)
Here $\tilde{K}\simeq \R$ is the connected subgroup of $\widetilde{SL}(2,\R)$ with Lie algebra $\mathfrak{so}(2,\R)$.  The action of $(a,b)\in \widetilde{SL}(2,\R)\times \tilde{K}$ on $c\in  \widetilde{SL}(2,\R)$ is given by $c\mapsto acb^{-1}$.  The center of $\widetilde{SL}(2,\R)$ is isomorphic to $\Z$ and is contained in $\tilde{K}$.   The subgroup $D$ of $\widetilde{SL}(2,\R)\times \tilde{K}$ is the image of the center under the embedding $z\mapsto (z, z)\in \tilde{K}\times \tilde{K}<\widetilde{SL}(2,\R)\times \tilde{K}$.   Viewing the Iwasawa group $S_1$ as a subgroup of $\widetilde{SL}(2,\R)$, then $S_1\times \tilde{K}$ is a simply-transitive subgroup of $\Isom_0(\widetilde{SL}(2,\R), h)$ isomorphic to $S=S_1\times \R$.   Thus the metric $h$ defines a left-invariant Riemannian metric $h_2$ on $S$ with $$\Isom_0(S,h_2)\simeq(\widetilde{SL}(2,\R)\times \tilde{K})/D.$$   This completes the proof that $S$ does not admit a maximally symmetric left-invariant Riemannian metric.

\end{example}

\begin{remark}
The metric $h_1$ in Example~\ref{ex.sl} is a solvsoliton, i.e., a left-invariant Ricci soliton on the solvable Lie group.  Thus the Main Theorem~\ref{mt} fails when Einstein is replaced by Ricci soliton.
\end{remark}

\section{Background}\label{einstauts}

\begin{notconv} In the remainder of this article, we will always use the corresponding fraktur, with any appropriate subscripts or superscripts, to denote the Lie algebra of a given Lie group.  E.g., if a Lie group is named $G_1$, its Lie algebra will be denoted $\g_1$.
\end{notconv}

In this preliminary section, we first review existence and structural results for Einstein solvmanifolds of negative Ricci curvature.   We then discuss a technique of Y. Nikolayevsky for determining whether a solvable Lie group admits such a metric.   Finally we review the structure theory of isometry groups of arbitrary solvmanifolds.

\subsection{Solvable Lie groups admitting Einstein metrics of negative Ricci curvature}\label{subsec.einst}

We restrict our attention to non-flat homogeneous Einstein metrics.    Any solvable Lie group admitting such a non-flat Einstein metric is necessarily non-unimodular \cite{DottiMiatello:TransitveGroupActionsAndRicciCurvatureProperties}.

\begin{defin}\label{std}\text{}
\begin{enumerate}
	\item A solvable Lie group $S$ will be said to be of \emph{Einstein type} if it admits a left-invariant Einstein metric of negative Ricci curvature.   We will also say that its Lie algebra $\s$ is of Einstein type.   We will say that the nilradical $N$ of $S$ (or the nilradical $\n$ of $\s$) is an Einstein nilradical.
	\item A non-unimodular, metric solvable Lie algebra $(\s, \Ip)$ is said to be \emph{standard} if it can be written as an orthogonal direct sum $\s=\af+\n$ where  $\af$ is abelian and $\n=[\s,\s]$  is the nilradical.   
	\item A standard metric solvable Lie algebra is said to be of \emph{Iwasawa} type if $\ad(A)|_{\n}$ is symmetric and non-zero for all $A\in \af$ and if there exists some $H\in\af$ such that $\ad(H)|_{\n}$ is positive-definite.   We will say a solvable Lie algebra $\s$ is of Iwasawa type if it admits an inner product satisfying these conditions.
	\item A solvable Lie group (with a left-invariant Riemannian metric) is said to be standard, respectively of Iwasawa type, if the associated metric Lie algebra is standard, respectively of Iwasawa type.   

\end{enumerate}

\end{defin}

In 1998, J. Heber \cite{Heber} extensively addressed the structure of standard Einstein solvmanifolds.  Heber's work resulted in great interest in the question of whether all non-flat Einstein solvmanifolds are standard.   A decade later J. Lauret \cite{LauretStandard} answered this question in the affirmative.   The three propositions in this and the next subsection summarize the part of the extensive work of Heber, Lauret and Y. Nikolayevsky that provide the background needed in the later sections of this paper.

\begin{notarem}\label{nota.ss}\text{}
\begin{enumerate}
\item Given a decomposition $\s=\af+\n$ of a solvable Lie algebra $\s$, where $\n=\nilrad(\s)$ and $\af$ is an abelian complement, we will denote by $\Der(\s)^{\af}$ the space of derivations commuting with $\ad_\s(\af)$, i.e., the derivations that vanish on $\af$.  Via the isomorphism $\alpha\mapsto \alpha|_{n}$, we may identify $\Der(\s)^\af$ with the space $\Der(\n)^\af$ of all derivations of $\n$ commuting with $\ad(\af)|_\n$.

\item  If $T$ is a semisimple endomorphism of a finite-dimensional vector space $V$, then $T$ can be uniquely decomposed as $T=T^\R + T^{i\R}$ where $T^\R$, respectively $T^{i\R}$, is a semisimple endomorphism with all eigenvalues real, respectively, purely imaginary, and where $T^\R$ and $T^{i\R}$ commute with $T$ (hence with each other).  Moreover, $T^\R$ and $T^{i\R}$ commute with any other endomorphism that commutes with $T$.  We will refer to $T^\R$ and $T^{i\R}$ as the \emph{real and imaginary parts} of $T$.   If $D$ is a semisimple derivation of a Lie algebra $\g$, then $D^\R$ and $D^{i\R}$ are also derivations of $\g$.
\end{enumerate}
\end{notarem}

\begin{prop}[Heber \cite{Heber}]\label{prop.heb} 
If a solvable Lie group $S$ admits a standard Einstein metric $g$, then \emph{every} Einstein metric on $S$ is isometric to $g$ up to scaling.   

Let $\s=\af+\n $ be a decomposition of $\s$ as in Definition~\ref{std}.  Then in the notation of~\ref{nota.ss}:
\begin{enumerate}
\item    
$\Der(\s)=\ad_\s(\n) + \Der(\s)^{\af}.$
Moreover, $\Der(\s)^{\af}$ is reductive and decomposes as $\kf+\p$ where the elements of $\kf$ are skew-symmetric and the elements of $\p$ are symmetric relative to $g$.  (We will often identify elements of $\Der(\s)^{\af}$ with their restrictions to $\n$.)
\item For $0\neq A\in\af$, we have $0\neq \ad(A)^\R\in\p$ and $\ad(A)^{i\R}\in \kf$.   Let $\af'=\{\ad(A)^\R:A\in \af\}$ and let $\s'$ be the semi-direct sum of $\af'$ and $\n$.   Then $\s'$ is of Iwasawa type and the associated simply-connected solvable Lie group acts simply-transitively on $S$.  
\item Let $H$ be the unique element of $\af$ such that $g(H, A)=\mytrace(\ad(A))$ for all $A\in\af$.   Then there exists $\lambda\in\R^+$ such that the eigenvalues of $\lambda\, \ad(H)^\R|_\n$ are positive integers.  Thus $\s'$ is of Iwasawa type.  The derivation $\ad(H)|_{\n}^\R$ is sometimes called the \emph{Einstein derivation}.
\item  Let $\cf$ be any abelian subspace of $\p$  containing the Einstein derivation.   Then the semi-direct product $\cf\ltimes \n$ of $\cf$ with $\n$ admits an Einstein inner product.   
\end{enumerate}
\end{prop}

\begin{defin} A left-invariant Riemannian metric $g$ on a nilpotent Lie group $N$ is called a \emph{nilsoliton} if it is a Ricci soliton.   This condition is equivalent to $Ric_g =c Id + D$ for some constant $c$ and some $D\in \Der(\n)$.   (Here $Ric_g$ is the Ricci operator.)
\end{defin}

\begin{remark}  In the defintition above, the condition $Ric_g = c Id +D$ always produces a left-invariant Ricci soliton on a Lie group.  In the case of nilpotent groups, all Ricci solitons satisfy this condition, but this is not true more generally for solvable Lie groups.  If a left-invariant metric on a solvable Lie group satisfies $Ric_g = c Id +D$, it is called a solvsoliton.  

It is known that a Ricci soliton on a solvable Lie group is isometric to a solvsoliton (on a possibly different solvable Lie group) \cite{Jablo:HomogeneousRicciSolitons}.  To go between these two different solvable Lie structures, one uses the process of `modifications', see Definition \ref{def.stdpos}.


\end{remark}

\begin{prop}[Lauret \cite{LauretStandard,LauretNilsoliton}]\label{prop.Lau}
$\text{}$
\begin{enumerate}
\item Every Einstein solvmanifold of negative Ricci curvature is standard.
\item Let $N$ be a simply-connected nilpotent Lie group.   Then $N$ is an Einstein nilradical if and only if $N$ admits a nilsoliton metric.   If $(S,g)$ is any Einstein solvmanifold of negative Ricci curvature with nilradical $N$, then the restriction of $g$ to $N$ is a nilsoliton.
\item A nilpotent Lie group admits at most one nilsoliton metric, up to automorphism and scaling.
 \end{enumerate}
\end{prop} 

\begin{remark}\label{rem.maxred}  Let $\s$ be a solvable Lie algebra of Einstein type.  By Propostion~\ref{prop.Lau}, every Einstein inner product on $\s$ is standard.  Given such an inner product, write $\s=\af+\n$, where $\af=\n^\perp$.   By Proposition~\ref{prop.heb}, $\ad_\s(\af)$ is a fully reducible subalgebra of $\ad(\s)$.  Moreover, $[\af,\n]=\n$, so $\s$ has trivial center and $\ad_\s(X)$ is a non-trivial nilpotent operator for every $X\in\n$.   Thus $\ad_\s(\mathfrak a)$ is a maximal $ad$-reductive subalgebra.    By the work of Mostow \cite[Theorem 4.1]{Mostow:FullyReducibleSubgrpsOfAlgGrps}, all maximal fully reducible subalgebras of linear Lie algebras, in particular, of $\ad(\s)$, are conjugate by an inner automorphism.    Since $\mathfrak s$ has no center, $ad: \mathfrak s \to \mathfrak{gl(s)}$ is an isomorphism onto its image.  Thus, the decomposition $\s=\af+\n$ is unique up to conjugacy by an element of the nilradical and every maximal fully $\ad$-reducible subalgebra of $\s$ is conjugate to $\af$.

If $\af$ is any maximal fully $\ad$-reducible subalgebra of $\s$, we will refer to $\s=\af+\n$ as a \emph{standard decomposition} of $\s$.  By the uniqueness statement above, given any standard decomposition $\s=\af+\n$ of a solvable Lie algebra of Einstein type, there exists an Einstein metric for which $\af\perp\n$.
\end{remark}

\subsection{Nikolayevsky's technique}\label{subsec.nik}
 
\begin{defin}\label{def.pre-Einst} (See  \cite{Nikolayevsky:EinsteinSolvmanifoldsandPreEinsteinDerivation}.)  
 A derivation $\varphi \in Der(\mathfrak g)$ of a Lie algebra $\g$ is a \textit{pre-Einstein derviation} if it is semisimple as an element of $End(\mathfrak g)$ with all eigenvalues real, and satisfies
\begin{equation}\label{eqn:pre-Einstein deriv} \mytrace(\varphi A) = \mytrace(A)  \quad \mbox{ for all } A\in \Der(\g)
\end{equation}

 \end{defin}

 \begin{prop}[Nikolayevsky \cite{Nikolayevsky:EinsteinSolvmanifoldsandPreEinsteinDerivation}]\label{prop.nik}
 $\text{}$
\begin{enumerate}
 \item Every Lie algebra admits a pre-Einstein derivation $\varphi$, unique up to automorphism.   The eigenvalues of $\varphi$ are rational.
 \item If $N$ is an Einstein nilradical, then the Einstein derivation (see Proposition~\ref{prop.heb}(iii)) of every Einstein solvmanifold with nilradical $N$ is a positive multiple of a pre-Einstein derivation $\varphi$ of $\n$.
 
\end{enumerate}
 
 \end{prop}

  \begin{prop}[Nikolayevsky \cite{Nikolayevsky:EinsteinSolvmanifoldsandPreEinsteinDerivation}]\label{prop.nik2}
Let $\n$ be a nilpotent Lie algebra of dimension $n$.   View the bracket $\mu:\n\times\n\to\n$ as an element of $V=\wedge^2(\mathbb R^n)^*\otimes \mathbb R^n$.  The group $GL_n(
\R)$ acts on $V$ via $A.\nu(x,y)=A\nu(A^{-1}x,A^{-1}y)$ for $A\in GL_n(\R)$, $\nu\in V$ and $x,y\in\R^n$, and this action gives rise to an action of the Lie algebra $\mathfrak{gl}_n(\R)$ on $V$.
Fix a choice of pre-Einstein derivation $\varphi$ of $\n$.  Let $t: GL_n(\R)\to \R$ be given by $t(A)=\mytrace(A\varphi)$ and let 
$$\mathfrak g_\varphi = \mathfrak{sl}_n(\mathbb R) \cap \mathfrak z(\varphi) \cap Ker~t$$  
where $\mathfrak z(\varphi)$ is the centralizer of $\varphi$ in $\mathfrak{gl}_n(\R)$.   Let $G_\varphi$ be the connected subgroup of $SL_n(\R)$ with Lie algebra $\mathfrak g_\varphi$.   The group $G_\varphi$ is fully reducible, and
the simply-connected Lie group $N$ with Lie algebra $\n$ admits a nilsoliton metric if and only if the orbit $G_\varphi.\mu$ is closed in $V$.
  
  \end{prop}

 \begin{notarem}\label{note.nik}\text{}  
 \begin{enumerate} 
\item Observe that the Lie algebra of the stabilizer of $\mu$ in $G_\varphi$ is precisely $\Der(\n)\cap \mathfrak{sl}(n,\R)$. 
\item  As discussed in \cite{Nikolayevsky:EinsteinSolvmanifoldsandPreEinsteinDerivation}, the group $G_\varphi$ is pre-algebraic; it is the identity component of the algebraic, fully reducible subgroup  $\Gt_\varphi < SL(n,\R)$ given as follows:  Let $V_1,\dots, V_k$ be the eigenspaces of $\varphi$ and $\lambda_1,\dots, \lambda_k$ the corresponding eigenvalues.  By~\ref{prop.nik} the $\lambda_j$ are positive rational numbers.  Let $N$ be the least positive integer for which all the $a_j:=N\lambda_j$ are integers.  The group  $\Gt_\varphi $ is defined by
 $$\Gt_\varphi=  \{(\alpha_1,\dots, \alpha_k)\in\prod_{j=1}^k\,GL(V_j)<GL(\n): \prod_{j=1}^k\,\det(\alpha_j)= \prod_{j=1}^k\,\det(\alpha_j)^{a_j}=1\}.$$
 
 \item The subgroup $G_\varphi$ of $SL(n,\R)$ is self-adjoint with respect to any inner product for which $\varphi$ is symmetric.
\item  We will sometimes abuse language and identify the bilinear map $\mu$ in Proposition~\ref{prop.nik2} with the Lie algebra $\mathfrak n$.   For $X\in\R^n$, we will write $\ad_\mu(X):\R^n\to\R^n$ to mean the linear mapping of $\R^n$ associated with the bracket $\mu$.

\end{enumerate}

   \end{notarem} 
   
 We  conclude this section with a corollary of Propositions~\ref{prop.heb}, \ref{prop.Lau}, and \ref{prop.nik}.   

\begin{cor}\label{cor.nik}
Let $N$ be a simply-connected nilpotent Lie group that admits a nilsoliton metric, and let $\varphi$ be the pre-Einstein derivation of $\n$.  Let $\af$ be any abelian subalgebra of $\Der(\n)$, all of whose elements are semisimple with non-zero real part, and suppose that the pre-Einstein derivation $\varphi$ is given by $A^\R$ for some $A\in \af$.   Then $\af\ltimes\n$ is a solvable Lie algebra of Einstein type.
\end{cor}

Since the corollary requires a slight strengthening of Proposition~\ref{prop.heb}(iv), we include the proof here after first recalling results of Mostow concerning Cartan decompositions.   For later use, we state Mostow's results in greater generality than needed for the proof of Corollary~\ref{cor.nik}.

\begin{notarem}\label{cd} In the language of Mostow \cite{Mostow:SelfAdjointGroups}, an ``fcc'' group is a Lie group with finitely many connected components.   If $\hat{H}$ is an fcc group, every maximal compact subgroup $\hat{K}$ of $\hat{H}$ satisfies $\hat{H}=H\hat{K}$ where $H$ is the identity component of $\hat{H}$.   In particular, $K:=\hat{K}\cap H$ is a maximal compact subgroup of $H$ and $\hat{K}/K$ is finite.  Any two maximal compact subgroups of $\hat{H}$ are conjugate via an element of $H$.    Generalizing the language of semisimple Lie group theory, we say that $\hat{H}=\hat{K}P$ is a \emph{Cartan decomposition} of $\hat{H}$ if: (i) $\hat{K}$ is a maximal compact subgroup of $\hat{H}$ and (ii) there exists a compact real form $\cf$  of the complexification $\h^\C$ of $\h$ such that $\kf=\h\cap \cf$ and $P=\exp(\p)$, where $\p=\h\cap i\cf$.  (Here $i=\sqrt{-1}$, and $\h$ and $\kf$ are the Lie algebras of $\hat{H}$ and $\hat{K}$, respectively.)  The existence of a Cartan decomposition implies that $\hat{H}$ is a reductive Lie group.  However, not every reductive fcc group admits a Cartan decomposition.   If a Cartan decomposition does exist, it is unique up to conjugation by elements of the identity component $H$ of $\hat{H}$.

Theorem 4.1 of \cite{Mostow:SelfAdjointGroups} states:   If $\hat{H}$ and $\hat{H}'$ are fcc Lie groups with $\hat{H}<\hat{H}'$ and if both $\hat{H}$ and $\hat{H}'$ admit Cartan decompositions, then for each Cartan decomposition $\hat{H}=\hat{K}P$, there exists a compatible Cartan decomposition  $\hat{H}'=\hat{K}'P'$.  By \emph{compatible} we mean that $\hat{K}<\hat{K}'$ and $P<P'$.    
\end{notarem}

\begin{proof}[Proof of Corollary~\ref{cor.nik}] Let $\af_0$ be the one-dimensional Lie algebra $\R\varphi$ where $\varphi$ is the pre-Einstein derivation of $\n$.   By Propositions~\ref{prop.Lau} and~\ref{prop.nik}, $\s_0:=\af_0\ltimes\n$ is a solvable Lie algebra of Einstein type.  Let $g_0$ denote both an Einstein metric on $\s_0$ and its restriction to a nilsoliton metric on $\n$.  By~\ref{nota.ss} and the hypotheses that $\af$ is abelian and that $\varphi=A^\R$ for some $A\in\af$, we have $\af<\Der(\n)^{\af_0}$.  Again by~\ref{nota.ss}, $X^\R$ and $X^{i\R}\in \Der(\n)^{\af_0}$ for all $X\in\af$.   
Let $\bfrak=\af^\R+\af^{i\R}$, where $\af^\R=\{X^\R:X\in\af\}$ and similarly for $\af^{i\R}$.  Note that $\bfrak=\af^{i\R}+\af^{\R}$ is a Cartan decomposition of $\bfrak$.   By~\ref{cd}, there exists a Cartan decomposition $\kf+\p$ of $\Der(\n)^{\af_0}$ such that $\af^{i\R}\subset \kf$ and $\af^\R \subset \p$.   By the conjugacy of Cartan decompositions and Proposition~\ref{prop.heb}, there exists $\tau\in \Aut(\n)^{\af_0}$ such that the elements of $\kf$, respectively $\p$, are skew-symmetric, respectively symmetric, with respect to the nilsoliton metric $\tau^*(g_0)$ on $\n$.   Since $\tau(\varphi)=\varphi$, the Einstein derivation of $(\s_0,g)$ is again $\varphi$.  By Proposition~\ref{prop.heb}, $\rf:=\af^\R+\n$  is of Einstein type.  Let $M$ be the associated simply-connected Einstein manifold.   
Since $\af^{i\R}$ acts skew-symmetrically, the isometry algebra of $(R,g)$ contains $\af^{i\R}+\rf=\bfrak+\n$, and it is easy to see that the simply-connected Lie group $S$ with Lie algebra $\af+\n<\bfrak+\n$ acts simply-transitively on $M$.  Thus $\af+\n$ is of Einstein type.

\end{proof}

\subsection{Isometry groups of solvmanifolds}\label{subsec.isom}

We review results of \cite{GordonWilson:IsomGrpsOfRiemSolv} concerning the structure of isometry groups of arbitrary solvmanifolds.   We will restrict our attention here to simply-connected solvmanifolds, since all solvable Lie groups of Einstein type are simply-connected.   

\begin{defin}\label{def.stdpos} Let $(\mathcal{M}, h)$ be a simply-connected Riemannian solvmanifold, let $G=I_0(\mathcal{M})$ be the identity component of the full isometry group of $\mathcal{M}$, and let $\g$ be the Lie algebra of $G$.  Let $\Rcal=\Rcal(h)$ denote the collection of all simply-transitive solvable subgroups of $G$.   Fix once and for all a base point $p\in \Mcal$.   For $R\in\Rcal$, we will continue to denote by $h$ the left-invariant Riemannian metric on $R$ defined by identifying $R$ with $\Mcal$ via $a\mapsto a(p)$ for $a\in R$.
\begin{enumerate}
\item    For $R\in\Rcal$, recall that the normalizer $ N_\g(\rf)$ of $\rf$ in $\g$ is given by the semi-direct sum $N_\g(\rf)=\Der_{\operatorname{skew}}(\rf,h)+\rf$ where $\Der_{\operatorname{skew}}(\rf,h)$ is the space of skew-symmetric derivations of $(\rf,h)$.    The \emph{standard modification} $\rf'$ of $\rf$ with respect to $h$ is defined to be the orthogonal complement in $N_\g(\rf)$ of $\Der_{\operatorname{skew}}(\rf,h)$ with respect to the Killing form.   The connected subgroup $R'$ of $G$ with Lie algebra $\rf'$ will also be called the standard modification of $R$.  Observe that $R'\in\Rcal$.

\item We say that $R$ (or its Lie algebra $\rf$) is in \emph{standard position} in $G$ with respect to $h$ if it is equal to its own standard modification.

\end{enumerate}
\end{defin}

\begin{notarem}\label{remf} (See \cite{GordonWilson:IsomGrpsOfRiemSolv}.)  
\begin{enumerate}
\item Let $\mathcal{F}$ be the collection of subgroups of $G$ that are maximal with respect to the property of having no non-trivial connected noncompact simple subgroups.  Then the elements $F\in \mathcal{F}$ form a conjugacy class of subgroups of $G$ given as follows:  Let $G=G_1G_2$ be any Levi decomposition of $G$ and write $G_1=G_{nc}G_c$, where $G_{nc}$ and $G_c$ are the maximal semsimple connected normal subgroups of $G$ of noncompact and compact type, respectively.  Let $G_1= K_1A_1N_1$ be any Iwasawa decomposition of $G_1$ (in particular, $G_c<K_1$), and let $M_1$ be the centralizer of $A_1$ in $K_1\cap G_{nc}$.  Set
\begin{equation}\label{eqf}F=(M_1A_1N_1)G_cG_2.\end{equation}
Then $F\in\mathcal{F}$, and every element of $\mathcal{F}$ is of this form.
\item A subgroup $S_1$ of $G$ of the form $S_1=A_1N_1$, where $K_1A_1N_1$ is an Iwasawa decompostion of some semisimple Levi factor of $G$, will be called an \emph{Iwasawa subgroup} of $G$.
\item In the notation of (i), the group $K_1$ is compact if and only if $G_{nc}$ has finite center.  This condition is equivalent to the condition that the Lie algebra of some, hence any, maximal compact connected subgroup of $G$ is a maximal compactly embedded subalgebra of $\g$.  In this case, every maximal compact subgroup $U$ of $G$ is of the form $U=K_1(U\cap G_2)$ relative to some Levi and Iwasawa decompositions as in (i).  
\end{enumerate}
\end{notarem}

\begin{prop}\label{prop.f}\cite{GordonWilson:IsomGrpsOfRiemSolv}  Let $\mathcal{M}$ be a simply-connected Riemannian solvmanifold and let $G=I_0(\mathcal{M})$ be the identity component of the full isometry group of $\mathcal{M}$.  Then:
\begin{enumerate}
\item The collection of all simply transitive solvable subgroups of $G$ in standard position with respect to $h$ forms a (non-empty) conjugacy class $\mathcal{S}=\Scal(h)$ of subgroups of $G$.
\item  Given $R\in\Rcal$, let $R'$ be the standard modification of $R$ and let $R''$ be the standard modification of $R'$.   Then $R''$ is in standard position in $G$ with respect to $h$, and the normalizer of $R''$ in $G$ contains that of $R$.   

\item (See the notation of~\ref{remf}.)  For $S\in\mathcal{S}$, the normalizer $N_G(S)$ is an element of $\mathcal{F}$. Conversely, given any $F\in \mathcal{F}$, there exists  $S\in\mathcal{S}$ such that $N_G(S)=F$.   

\item Let $S\in\mathcal{S}$ and let $G=G_1G_2$ and $G_1=K_1A_1N_1$ be the Levi and Iwasawa decompositions associated with $F=N_G(S)$ as in~\ref{remf}.  Then the Lie algebra of $S$ satisfies $$\af_1+\n_1+[\g,\g_2]\subset \s\subset Z(\mf_1)+  \af_1 +\n_1+\g_2,$$ where $Z(\m_1)$  is the center of the Lie algebra $\m_1$ of $M_1$.   In particular, $S$ contains an Iwasawa subgroup of $G$.
  
\end{enumerate}
\end{prop}

\section{The Main Theorem}\label{main}

Our goal is to prove the following theorem:

\begin{thm}\label{mainthm}
Let $R$ be a solvable Lie group of Einstein type and $h$ a left-invariant Riemannian metric on $R$.   Let $\Gt=\Isom(R,h)$ be the full isometry group  and let $\Lt$ be the isotropy subgroup of $\Gt$ at the identity $e\in R$.   Then $\Gt/\Lt$ admits a $\Gt$-invariant Einstein metric.

\end{thm}

Since $R$ acts simply-transitively on $\Gt/\Lt$, it will follow that $R$ admits an Einstein metric whose isometry group contains that of $h$, thus proving the Main Theorem~\ref{mt}.

The proof of Theorem~\ref{mainthm} has three main parts:

\begin{enumerate}
\item We apply a result of the second author to show that the normalizer of $R$ (or of any simply transitive solvable subgroup of $\Gt$ in $\Gt$)  leaves an Einstein metric invariant.   This result along with a study of the structure of the isometry groups of left-invariant metrics on solvable Lie groups of Einstein type enables us to reduce the theorem to Lemma~\ref{lem}. 

\item We use Ricci curvature computations to prove the ``if'' statement and the final statement of Lemma~\ref{lem}.
\item We apply Nikolayevsky's technique, as outlined in Subsection~\ref{subsec.nik} to prove the forward statement of Lemma~\ref{lem}.

\end{enumerate}

In this section we carry out the first two parts of the proof.  
\subsection{Part (i) of the proof.}\label{subsec.part1}

In the notation of Theorem~\ref{mainthm}, we first show that the normalizer of $R$ in $\Gt$ leaves an Einstein metric invariant.   Recall that the normalizer of $R$ in $\Gt$ is the group $\Aut_{\operatorname{orth}}(R,h)$ of orthogonal automorphisms of $(R,h)$. 
\begin{prop}\label{prop.step1}  Let $R$ be a solvable Lie group that admits an Einstein metric and let $h$ be any left-invariant metric on $R$.   Then there exists an Einstein metric $g$ on $R$ such that 
$$\Isom(R,h)\cap \Aut(R)\subset \Isom(R,g).$$
\end{prop}

\begin{remark} We note that the proposition above holds more generally for solvsolitons, with the same proof, but we will not need this fact.
\end{remark}

\begin{proof} We apply results of \cite{Jablo:ConceringExistenceOfEinstein}.  (That article addresses the more general setting of solvable Lie groups admitting  solvsolitons but we only apply it to the special case of those admitting Einstein metrics.)  There is a natural correspondence between Einstein (or solvsoliton) metrics on $R$ and so-called \emph{distinguished metrics}.  Letting $\n$ denote the nilradical of $\rf$, a distinguished metric and the corresponding Einstein metric agree on $\n$ (and restrict to a nilsoliton metric on $\n$), and the orthogonal complement of $\n$ relative to both metrics is the same abelian algebra $\af$.  The two metrics differ only on $\af$.  The expression for the distinguished metric on $\af$ will not be needed below.   The Einstein metric is given on $\af$ by
	\begin{equation}\label{eq.la}g(A,A)=c~\mytrace(S_A)^2 
	\end{equation}
where $c$ is a constant and where $S_A$ is the symmetric part of $\ad(A)|_{\n}$ with respect to the given nilsoliton metric on $\n$; i.e., $S_A=\frac{1}{2}(\ad(A)|_{\n}+\ad(A)|_{\n}^t)$.

Theorem 4.1 of \cite{Jablo:ConceringExistenceOfEinstein} (stated as Proposition~\ref{completely_solv} above) states that left-invariant  solvsoliton metrics on completely solvable unimodular Lie groups are maximally symmetric.   In our case, $R$ is not assumed to be either completely solvable or unimodular.   However, the first step in the proof of Theorem 5.8 of \cite{Jablo:ConceringExistenceOfEinstein} applies to \emph{all} solvable Lie groups that admit solvsoliton metrics.  It asserts that for any left-invariant metric $h$ on $R$, there exists a distinguished metric $g_0$ such that 
	$$\Aut_{\operatorname{orth}}(R,h)\subset \Isom(R,g_0).$$
Here $\Aut_{\operatorname{orth}}(R,h)$ denotes the group $\Aut (R) \cap \Isom(R,h)$.  As $R$ is simply-connected, this group corresponds precisely to the orthogonal automorphisms of the Lie algebra, $\Aut(\rf)\cap O(\rf,h)$.

To complete the proof, we need only show that $\Aut_{\operatorname{orth}}(R, g_0)\subset \Aut_{\operatorname{orth}}(R,g)$ where $g$ is the Einstein metric associated with $g_0$.    
Let $\tau\in \Aut(\rf)\cap O(\rf,g_0)$, and let $A\in\af$.  By Proposition~\ref{prop.heb}, we have $\ad(A)|_{\n}=S_A + T_A$, where $T_A$ is a skew-symmetric derivation of $\n$ and $S_A$, as defined above, is a symmetric derivation.  Moreover, $S_A$ and $T_A$ both commute with $\ad(A)$ and hence with each other. 
 Since $\tau\in \Aut(\rf)$, we have $\tau|_{\n}\circ \ad(A)|_{\n}=\ad(\tau(A))|_{\n}\circ \tau|_{\n}$, so $\ad(\tau(A))|_{\n}=\tau|_{\n}\circ \ad(A)|_{\n}\circ \tau|_{\n}^{-1}$.  Since $\tau|_{\n}$ is orthogonal, we also have that $S_{\tau(A)}=\tau|_{\n}\circ S_A\circ \tau|_{\n}^{-1}$. 
It is now immediate from 
Equation~(\ref{eq.la}) that  $\tau\in O(\rf,g)$.    
\end{proof} 

Restricting our attention to the identity component $G$ of $\Gt$ for now, we next apply Proposition~\ref{prop.step1} to describe the subgroups of $G$ in standard position, in the language of Definition~\ref{def.stdpos}. 

\begin{lemma}\label{lem.stdmodEinst} Let $R$ be a solvable Lie group of Einstein type and let $h$ be an arbitrary but fixed left-invariant metric on $R$ (not necessarily an Einstein metric).  Let $G=I_0(R,h)$ and let $L$ be the isotropy subgroup of $G$ at the identity element.   We use the notation of~\ref{remf} and Proposition~\ref{prop.f}.   Then:
\begin{enumerate}
\item Each $S\in\mathcal{S}(h)$ is a solvable Lie group of Einstein type.  
\item There exist a Levi decomposition $G=G_1G_2$ and an Iwasawa decomposition $G_1=K_1A_1N_1$ such that, in the notation of~\ref{remf},
$$L=K_1(L\cap G_2).$$
In particular, $K_1$ has finite center and $L$ is a maximal compact subgroup of $G$.
\item There exists a characteristic subgroup $S_2$ of $G$ contained in $G_2$ such that:
\begin{enumerate}
\item $\mathcal{S}(h)=\{S_1\ltimes S_2: S_1\mbox{\,is an Iwasawa subgroup of \,}G\}.$
\item $G_2=(L\cap G_2)\ltimes S_2$.
\item $\nilrad(G)=\nilrad(S_2)$.
\end{enumerate}

\item Let $F\in\mathcal{F}$.  Then $F/(F\cap L)$ admits a left-invariant Einstein metric.   (In particular, the Einstein metric on $R$ and on the associated element $S$ of $\Scal(h)$ can be chosen to be invariant under $N_G(S)\in\mathcal{F}$.)

\end{enumerate}
\end{lemma}

\begin{proof} (i) Write $\rf=\af+\n$ as in Definition~\ref{std}.  Let $\h=\Der_{\operatorname{skew}}(\rf,h) +\rf$ (semi-direct product) and let $H$ be the semi-direct product of $R$ with the group $\Aut_{\operatorname{orth}}(R,h)$ of orthogonal automorphisms of $(R,h)$.  Then $H$ has Lie algebra $\h$.  By Proposition~\ref{prop.step1}, there exists an Einstein metric $g$ on $R$ invariant under $\Aut_{\operatorname{orth}}(R,h)$.  Thus $H/\Aut_{\operatorname{orth}}(R,h)$ admits an $H$-invariant Einstein metric.  Since $R'$ acts simply-transitively on $H/\Aut_{\operatorname{orth}}(R,h)$, this Einstein metric defines a left-invariant Einstein metric on $R'$.   Thus $R'$ is of Einstein type.  Continuing, the standard modification $S$ of $R'$ is also of Einstein type and, by Proposition~\ref{prop.f}, $S\in \Scal(h)$. 

(ii) We have $\g=\lf+\rf$ with $\lf\cap\rf=\{0\}$.  Let $\uf$ be a maximal compactly embedded subalgebra of $\g$ containing $\lf$.   Suppose $X\in \uf\cap\rf$.   Then $\ad(X)$ is semisimple with purely imaginary eigenvalues.  By Propositions~\ref{prop.heb}(ii) and~\ref{prop.Lau}, it follows that all eigenvalues of $\ad(X)$ are zero and thus $X$ must be central.  However, by Propositions~\ref{prop.heb}(iii) and ~\ref{prop.Lau}, the center of $\rf$ is trivial.  Thus $\uf\cap\rf=\{0\}$, and $\lf$ is a maximal compactly embedded subalgebra of $\g$.  Statement (ii) now follows from~\ref{remf}(iii).

(iii)  Define $\s_2$ to be the orthogonal complement of $\lf\cap\g_2$ in $\g_2$ with respect to the Killing form $B_\g$,
and let $S_2$ be the corresponding connected subgroup of $G$.  Then $\nilrad(\g)<\s_2<\g_2$.  Since $[\g,\g_2]<\nilrad(\g)$, it follows that $\s_2$ is a $\g$-ideal.  We first show that it is a characteristic ideal; i.e., that it is invariant under $\Aut(\g)$.  By statement (ii) and the fact that $B_\g(\g_1,\g_2)=0$ for any semisimple Levi factor $\g_1$, we see that $\s_2=\lf^\perp\cap \g_2$ where $\lf^\perp$ is the orthogonal complement of $\lf$ with respect to $B_\g$.  Let $\tau\in\Aut(G)$.  Then $\tau(L)$ is a maximal compact subgroup of $G$ and hence is conjugate to $L$; i.e., $\tau_*(\lf)=\Ad(a)(\lf)$ for some $a\in G$.    For $X\in\s_2$, we have 
$$B_\g(\lf, \tau_*(X))=B_\g(\tau^{-1}_*(\lf),X)=B_\g(\Ad(a^{-1})(\lf),  X)=B_\g(\lf,\Ad(a)(X))=0$$
where the last equality uses the fact that $\s_2$ is a $\g$-ideal and thus is invariant under $\Ad(G)$.   Thus 
$\tau_*(\s_2)\perp \lf$ with respect to $B_\g$.  Since also $\tau_*(\s_2)<\g_2$, we have $\tau_*(\s_2)<\s_2$ and trivially equality must hold.   Thus $\s_2$ is a characteristic ideal in $\g$ and $S_2$ is a characteristic subgroup of $G$.

The fact that the Killing form of $\g$ is negative-definite on $\lf$ implies that $\g_2=(\lf\cap\g_2)\ltimes \s_2$.  Thus $S_2$ satisfies condition (b).

 For (c), since $\nilrad(\g)$ is a nilpotent ideal of $\s_2$, we have $\nilrad(\g)<\nilrad(\s_2)$.  For the opposite inclusion, note that any subspace of $\g_2$  containing $\nilrad(\g_2)$ is a $\g$-ideal since $[\g,\g_2]<\nilrad(\g)$.  In particular, $\nilrad(\s_2)$ is a nilpotent ideal of $\g$ and hence $\nilrad(\s_2)<\nilrad(\g)$.  Thus $S_2$ satisfies condition (c).

Finally we prove that condition (a) holds.  Consider the Levi and Iwasawa decompositions in part (ii) of the Lemma and the corresponding group $F\in \mathcal{F}$ given by $F=(M_1A_1N_1)G_cG_2$.  By (ii), $L\cap F=M_1G_c(L\cap G_2)$.   Let $S\in\Scal(h)$ be the subgroup of $F$ in standard position, i.e., $\s$ is the orthogonal complement of $\lf\cap \mathfrak{f}$ in $\mathfrak{f}$ relative to $B_\mathfrak{f}$.  For $X\in\lf\cap \mathfrak{f}$ and $Y\in\g_2$, we have 
$$B_\mathfrak{f}(X,Y)=\mytrace(\ad(X)\circ\ad(Y)_{|\g_2})=B_\g(X,Y).$$
It thus follows from the definition of $\s_2$ that $\s_2\perp(\lf\cap\mathfrak{f})$ relative to $B_\mathfrak{f}$ and hence $\s_2<\s$ by Definition~\ref{def.stdpos}.  By Proposition~\ref{prop.f}, we also  have that $\af_1+\n_1<\s$.  Write $\s_1=\af_1+\n_1$.   Since $\mathfrak{f}=(\lf\cap\mathfrak{f})+(\s_1+\s_2)$, we must have $\s=\s_1+\s_2$, and then $S=S_1\ltimes S_2$.   Thus we have found one element $S\in\Scal(h)$ of the form stated in condition (a).   Condition (a) now follows from Proposition~\ref{prop.f}(i), the fact that $S_2$ is normal in $G$ and the fact that the Iwasawa subgroups of $G$ form a $G$-conjugacy class of subgroups.

 (iv) is immediate from Proposition~\ref{prop.step1}.  

\end{proof}
\begin{remark} One can also show directly, using Proposition~\ref{prop.heb} and ~\ref{prop.Lau}, that the standard modification of $R$ with respect to $h$ is of Einstein type and moreover that it is given by $\rf'=\af'+\n$, where $\n$ is the nilradical of both $\rf$ and $\rf'$.   Moreover, by the proof of Theorem 3.5 of \cite{GordonWilson:IsomGrpsOfRiemSolv}, the fact that the standard modification $\rf'$ satisfies $\nilrad(\rf')=\nilrad(\rf)$ implies that $R'$ is in standard position with respect to $h$.   Thus for solvable Lie groups of Einstein type, only a single standard modification is needed to reach standard position.   We will not need this fact, however.
\end{remark}  

Lemma~\ref{lem.stdmodEinst} says that each $S\in\Scal(h)$ satisfies the hypotheses of the Key Lemma stated below.

\begin{key}[``Only if'' statement of Lemma~\ref{lem}.] \label{key}Suppose that $S$ is a solvable Lie group of Einstein type and that $S$ is a semi-direct product $S=S_1\ltimes S_2$ of subgroups satisfying the following hypotheses:
\begin{itemize} 
\item $S_1$ isomorphically embeds as an Iwasawa subgroup in a semisimple Lie group $G_1$.   
\item The adjoint action of $S_1$ on the Lie algebra $\operatorname{Lie}(S_2) $ extends to a representation of $G_1$ on $\operatorname{Lie}(S_2)$.
\end{itemize}
Then, $S_2$ is of Einstein type.\end{key}

In the remainder of this section, we assume the Key Lemma and complete the proof of the Main Theorem~\ref{mainthm}.  We will then prove the Key Lemma in a later section.

Assuming the Key Lemma , we have reduced the proof of the Main Theorem to the following proposition:

\begin{prop}\label{prop.key} Let $\Gt$ be a (not necessarily connected) Lie group, let $\Lt$ be a compact subgroup of $\Gt$ and denote by $G$ and $L$ the identity components of $\Gt$ and $\Lt$, respectively.    Suppose that there exists a Levi decomposition $G=G_1G_2$, an Iwasawa decomposition $G_1=K_1S_1$, and a connected solvable normal subgroup $S_2$ of $\Gt$ satisfying the following:
\begin{enumerate}
\item $L=K_1(L\cap G_2)$ and $G_2=(L\cap G_2)S_2$;
\item The solvable group $S:=S_1\ltimes S_2$ acts simply transitively on $\Gt/\Lt$;
\item $S_2$ is of Einstein type.
\end{enumerate}
Then $\Gt/\Lt$ admits a left-invariant Einstein metric of negative Ricci curvature.
\end{prop}

(There is some redundancy in the hypotheses of the proposition; one can show that hypothesis (i) follows from the remaining hypotheses.)  Note that this proposition together with the Key Lemma~\ref{key}  form Lemma~\ref{lem} stated in the introduction.  

 The Main Theorem~\ref{mainthm} is an immediate consequence of Lemma~\ref{lem.stdmodEinst}, Lemma~\ref{key} and Proposition~\ref{prop.key}.  The statement of Proposition~\ref{prop.key} is actually stronger than needed to prove the Main Theorem, since it does not assume that $S$ itself is of Einstein type; instead the conclusion of the proposition implies that $S$ is of Einstein type.

\begin{prop}\label{prop.full_aut} Let $S$ be a solvable Lie group of Einstein type and let $W$ be a maximal compact (not necessarily connected) subgroup of $\Aut(\s)$.  Then:
\begin{enumerate}
\item There exists a left-invariant Einstein metric $g$ on $S$ that is $W$-invariant.    
\item Let $H$ be the connected reductive subgroup of $\Aut(\s)$ with Lie algebra $\Der(s)^\af$, where $\s=\af+\n$ is the orthogonal decomposition of $\s$ given in Proposition~\ref{prop.heb} with respect to the Einstein metric $g$ in (i).  Then $H$ is normalized by $W$.   The Lie subgroup $\Ht:=W H$ of $\Aut(S)$ has identity component $H$ and has only finitely many connected components. 
\item In the language of~\ref{cd}, $\Ht$ has a Cartan decomposition $\Ht=WP$, with $W$ acting orthogonally and $P$ acting symmetrically on $\s$ with respect to the inner product $g$.
\end{enumerate} 

\end{prop}

\begin{proof} The first statement is immediate from Proposition~\ref{prop.step1}.   For the second statement, note that $H$ is the identity component of $\Aut(\s)^\af:=\{\tau\in \Aut(S): \tau|_{\af}=Id_\af\}$.   The group $W$ leaves $\af$ invariant since it acts orthogonally relative to $g$ and thus normalizes $H$.   The Lie algebra of $\Ht$ contains  
 $\Der(\s)^\af$ and consists of derivations that leave $\af$ invariant.  By Proposition~\ref{prop.heb}, it follows that $\Ht$ has Lie algebra $\h:=\Der(\s)^\af$.  This, together with the fact that $W$ is compact, yields (ii).

Proposition~\ref{prop.heb}(i) gives us a Cartan decomposition $H=WP$, where $P=\exp(\p)$, with $\p$ consisting of all elements of $\h=\Der(\s)^\af$ that are symmetric with respect to $g$.   Since $\hat{W}$ acts orthogonally with respect to $g$, this Cartan decomposition extends to a Cartan decomposition $\Ht=\hat{W} P$.

\end{proof}

\begin{ab}[{\bf Choosing an Einstein metric on $S_2$}]\label{lem.com}  In the notation of Proposition~\ref{prop.key}, the Lie group $\Gt$ acts by conjugation on $S_2$.  Let $\rho:\Gt\to \Aut(\s_2)$ be the resulting action on $\s_2$ so $\rho_*(X)=\ad(X)_{|\s_2}$.  By Proposition~\ref{prop.step1}, there exists a left-invariant Einstein metric $h_2$ on $S_2$ invariant under the action of $\Lt$.   We have a standard decomposition $\s_2=\af_2+\n_2$ as in Definition~\ref{std}, with  $\af_2\perp \n_2$ relative to $h_2$.  The group $\Lt$ normalizes $\af_2$ and $L$ acts trivially on $\af_2$.  By Proposition~\ref{prop.heb},  we have that $\g=\n_2 + Z_\g(\af_2)$, and $\rho_*(Z_\g(\af_2))$ lies in the reductive Lie algebra $\Der(\s_2)^\af$.    In fact $Z_\g(\af_2)$ is  itself a reductive Lie algebra complementary to $\n_2$.   Indeed, by Proposition~\ref{prop.heb}, we have  $Z_\g(\af_2)\cap\n_2)=\{0\}$ and thus the projection $\g\to \g/\n_2$ restricts to an isomorphism of $Z_\g(\af_2)$ with $\g/\n_2$.   But it is easily seen that $\n_2=\nilrad(\g)$.  Thus $\g/\n_2$, and hence $Z_\g(\af_2)$, is reductive.    The derived algebra of $Z_\g(\af_2)$ is a semisimple Levi factor of $\g$, we replace $\g_1$ by $[Z_\g(\af_2),Z_\g(\af_2)]$.    This may result in replacing $L$ (thus $K_1$) and $S_1$ by conjugates, but that does not affect the validity of Proposition~\ref{prop.key}.  

The Lie group $G_1\Lt$ is a reductive fcc group (see~\ref{cd}) with identity component $G_1(L\cap G_2)$.     Let $G_1=K_1P_1$ be a Cartan decomposition.   Noting that $K_1=L\cap G_1$, we see that $\Lt P_1$ is a Cartan decomposition of $G_1\Lt$.   Let $S_2$ play the role of $S$ in Proposition~\ref{prop.full_aut}.   Then the Lie group $\Ht$ in~\ref{prop.full_aut} contains $G_1\Lt$ and thus admits a Cartan decomposition $\Ht=WQ$ compatible with the Cartan decomposition $\Lt P$.   By Proposition~\ref{prop.full_aut} and the uniqueness of Cartan decompositions up to conjugacy, there exists $\tau\in\Ht$ such that $W$, hence $\Lt$, acts orthogonally and $Q$, hence $P$, acts symmetrically on $\s_2$ relative to the Einstein inner product $\tau^*h_2$.   We set $g_2=\tau^* h_2$.  

\end{ab}

\begin{ab}[{\bf Computing Ricci Curvature}]

We review the expression for the Ricci curvature on a Riemannian homogeneous space $G/K$.  Let $\g=\kf+\q$ be a reductive decomposition (I.e., $\q$ is an $\Ad_G(K)$-invariant complement of $\kf$), and let $\Ip$ be the Riemannian inner product on $\q$.  We may extend $\Ip$ to an $\Ad_G(K)$-invariant inner product on $\g$ with $\kf\perp\q$.  For $X\in\g$, write $X=X_\kf +X_\q$ with $X_\kf\in\kf$ and$X_\q\in\q$.  Let $\Rc:\q\times \q\to \R$ denote the Ricci tensor of $\Ip$, and let $\Ric:\q\to\q$ denote the corresponding Ricci operator; i.e., $\langle \Ric(X),Y\rangle =\Rc(X,Y)$ for $X,Y\in\q$.  Let $H\in\q$ be the unique element such that $\langle H,X\rangle =\mytrace(\ad(X))$ for all $X\in\g$.   Let $B_\g$ denote the 
Killing form of $\g$.  Define  $M:\q\to\q$ by
\begin{equation}\label{m}\langle M(X),Y\rangle=\sum_{i=1}^n\,-\frac{1}{2}\langle [X,X_i]_\q, [Y,X_i]_\q\rangle +\frac{1}{4}\sum_{i,j=1}^n\,\langle [X_i,X_j]_\q,X\rangle \langle [X_i,X_j]_\q,Y\rangle\end{equation}
where $\{X_1,\dots, X_n\}$ is an orthonormal basis of $\q$.  Then
\begin{equation}\label{ric}\Rc(X,Y) =\langle   M(X),Y\rangle -\frac{1}{2}B_\g(X,Y)-\frac{1}{2}\langle [H,X],Y\rangle -\frac{1}{2}\langle X, [H,Y]\rangle.\end{equation}

\end{ab}

\begin{proof}[Proof of Proposition~\ref{prop.key}]

Let  $\uf=Z(\af_2)=\lf +(\p_1+\af_2)$ and $\q=\p_1+\af_2+\n_2$.  Then $\uf$ is a reductive Lie algebra and $\g=\uf+\n_2=\lf+\q$ with each term in the two decompositions being $\Ad_{\Gt}(\Lt)$-invariant.  We define an inner product $\Ip$ on $\q$ satisfying:
\begin{enumerate}
\item $\p_1\perp\s_2$;
\item The restriction of $\Ip$ to $\s_2\times \s_2$ is the Einstein inner product defined in~\ref{lem.com};
\item Writing the noncompact part $\g_{nc}$ of $\g_1$ as a direct sum $\g_{nc}=\h_1\oplus\dots\h_r$ of simple ideals and letting $\p_{1,i}=\p_1\cap \h_i$, then $\p_{1,i}\perp\p_{1,j}$ for $i\neq j$;
\item The restriction of $\Ip$ to $\p_{1,i}\times\p_{1,i}$ is a positive multiple $\alpha_iB_{\h_i}$ of the Killing form of $\h_i$.
\end{enumerate}

Any such inner product is automatically invariant under $\Ad_G(L)$, where $L=K_1(L\cap G_2)$ is the identity component of $\Lt$.  Our goal is to choose the constants $\alpha_i$  
in such a way that $\Ip$ is $\Ad_{\Gt}(\Lt)$-invariant and so that the resulting Riemannian metric on $\Gt/\Lt$ is Einstein.

Let $\Rc:\q\times \q\to \R$ denote the Ricci tensor of $\Ip$, let $\Rc_1$ and $\Rc_2$ denote, respectively, the Ricci tensors of $G_1/K_1$ and $S_2$ with respect to the Riemannian metrics defined by the restrictions of $\Ip$ to $\p_1\times\p_1$ and to $\s_2\times\s_2$, and let $\Ric_1$ and $\Ric_2$ denote the corresponding Ricci operators.  Since $S_2$ is Einstein, we have $\Ric_2=c\,Id_{\s_2}$ for some negative constant $c$.  The metric on $G_1/K_1$ is symmetric and we have
\begin{equation}\label{rc1}\Ric_1=-\left(\frac{1}{\alpha_1}\Id_{\p_{1,1}}\times\dots\times \frac{1}{\alpha_r}\Id_{\p_{1,r}}\right).\end{equation}

We now compare the Ricci tensors $\Rc_1$ and $\Rc_2$ to the restictions of $\Rc$ to $\p_1$ and $\s_2$, respectively.   For $i=1,2$, we will write $H_i$ and $M_i$ for the expressions in Equation~\ref{ric} for $\Rc_i$.  Since $G_{1}$ is semisimple, and $\g_1\perp\s_2$, we have
\begin{equation}\label{h} H_1=0 \mbox{ and } H=H_2\in\s_2.\end{equation}

To compare $M$ and $M_i$, we use a computation carried out by Lauret and Lafuente in\cite{LauretLafuente:StructureOfHomogeneousRicciSolitonsAndTheAlekseevskiiConjecture}, Lemma 4.4.   They considered the case of a reductive decomposition $\g=\lf+\q$ with $\q=\h+\n$ where $\uf=\lf+\h$ is a reductive Lie algebra and $\n=\nilrad(\g)$.   In our notation, $\h=\p_1+\af_2$ and $\n=\n_2$.  For $X\in\g$, write 
$$\rho(X)=\ad(X)|_{\n_2}.$$   Using Lauret--Lafuente's computation, along with the fact that $\ad(X)|_{\n_2}$ is symmetric with respect to $\Ip$ for all $X\in\p_1$ (see~\ref{lem.com}), we find that 
\begin{equation}\label{m2}M_2=M|_{\s_2\times\s_2},\end{equation}  
\begin{equation}\label{m1}\langle M_1(X),Y\rangle=\langle M(X),Y\rangle+\frac{1}{2}\mytrace(\rho(X)\rho(Y)) \mbox{ for } X,Y\in\p_1\end{equation}
and 
\begin{equation}\label{m12}\langle M(X),Y\rangle=0 \mbox{ for } X\in\p_1 \mbox{ and } Y\in\s_2.\end{equation}
 (The third equation uses the fact that $\mytrace(\rho(X)\rho(Y))=0$ for $X\in\p_1$ and $Y\in\af_2$.)

The Killing forms satisfy 
\begin{equation}\label{kf}B_ \g|_{\s_2\times\s_2}=B_{\s_2} \mbox{ and } B_\g(X,Y)=B_{\g_{1}}(X,Y)+\mytrace(\rho(X)\rho(Y)) \mbox{ for } X,Y\in\p_1.\end{equation}
Equations~\ref{ric}--\ref{kf} yield
\begin{equation}\label{rc2}Rc_2=Rc|_{\s_2\times\s_2}\end{equation}
and
\begin{equation}\label{L}\Rc_1(X,Y)=\Rc(X,Y)+T(X,Y),\end{equation}
where 
\begin{equation}\label{rho}T(X,Y)=\mytrace(\rho(X)\rho(Y)) \mbox{ for } X,Y\in\g_{nc}.\end{equation}
Since $T$ is an $\Ad(G_{1})$ invariant bilinear form on $\g_{1}$, we have $T(\h_i,\h_j)=0$ for $i\neq j$ and, for each $i$, there exists a constant $\beta_i$ such that 
\begin{equation}\label{beta}T|_{\h_i\times\h_i}=\beta_iB_{\h_i}=\frac{\beta_i}{\alpha_i}\Ip|_{\p_{1,i}\times\p_{1,i}}.\end{equation}
Since $\rho|_{\p_1}$ is symmetric, we have $\beta_i\geq 0$ with equality if and only $[\h_i,\g_2]=0$.

We now define the constants $\alpha_i$, $i=1,\dots r$, in condition (iv) by
\begin{equation}\label{alpha}\alpha_i=\frac{-1-\beta_i}{c}\end{equation}
and observe that $\alpha_i>0$ since $c<0$ and $\beta_i\geq 0$.   By Equations~\ref{rc1}, \ref{rc2}, \ref{L}, \ref{beta} and \ref{alpha}, we have $\Rc=c\Ip$ on all of $\q$.   Thus we have constructed a left-invariant Einstein metric on $G/L$. 

It remains only to show that the inner product $\Ip$ on $\q$ is invariant under $\Ad_{\Gt}(\Lt)|_\q$.  Condition (ii) guarantees that the restriction of $\Ip$ to $\s_2\times \s_2$ is $\Ad_{\Gt}(\Lt)$-invariant, so it remains only to check the restriction to $\p_1$.  By~\ref{lem.com}, 
the Cartan decomposition $\g_1=\kf_1+\p_1$ is $\Ad_{\Gt}(\Lt)$-invariant.  Thus, for each $\gamma\in \Ad_{\Gt}(\Lt)$, there exists a permutation $\sigma$ of $\{1,2,\dots,r\}$ such that $\gamma(\h_{i})=\h_{\sigma(i)}$ and $\gamma(\p_{1,i})=\p_{1,\sigma(i)}$ for all $i$.   Since the automorphism $\gamma$ preserves both the bilinear form $T$ in Equation~\ref{rho} and intertwines the Killing form of $\h_i$ with that of $\h_{\sigma(i)}$, Equation~\ref{beta} shows that $\beta_{\sigma(i)}=\beta_i$, and then Equation~\ref{alpha} implies $\alpha_{\sigma(i)}=\alpha_i$ for each $i$.  Thus by Condition (iv), $\Ip$ is $\gamma$-invariant.
\end{proof}

\section{Technical lemmas on GIT} \label{sec: technical lemmas on GIT}

In this section, the groups of primary interest are fully reducible subgroups of $GL(V)$, where $V$ is a real or complex vector space.  For a point $p\in V$, we are interested understanding when the orbit $G\cdot p$ is closed in $V$.

\subsection{Preliminaries on fully reducible groups}  We first recall the definition of fully reducible.

\begin{defin}  A subgroup $G \subset GL(V)$ is called fully reducible if for any $G$-invariant subspace $W$ of $V$, there exists a complementary $G$-invariant subspace $W'$ of $V$.
\end{defin}

  Observe that if a subspace $W$ is invariant under $G$, then it is also invariant under the Zariski closure $\overline G$ of $G$.  In this way we see that $G$ being fully reducible implies $\overline G$ is also fully reducible.  Further, if $G$ is connected and fully reducible, then $G$ may be written as $G=[G,G] Z(G)$, where $[G,G]$ is semisimple, $Z(G)$ is the center of $G$, and the elements of $Z(G)$ are semisimple (i.e.\ diagonalizable).  This fact is well-known for algebraic groups and the proof is more or less the same in the case of connected Lie groups.


Let $H$ be an algebraic, fully reducible subgroup of an algebraic, fully reducible subgroup $G$ of $GL(V)$.  It is well-known that the normalizer $N_G(H)$ and centralizer $Z_G(H)$ of $H$ in $G$ are also fully reducible and algebraic.  
Further,  $Z_G(H) \cdot H $  is a finite index subgroup of  $N_G(H)$.  At the Lie algebra level, this is precisely
	$$N_\mathfrak g (\mathfrak h) = Z_\mathfrak g(\mathfrak h) + \mathfrak h.$$

\begin{defin}  Let $k=\mathbb R$ or $\mathbb C$ and consider the multiplicative group $k^*$ of all non-zero elements.
A \textit{1-parameter subgroup}  of $G$ is a homomorphism $\lambda : k^* \to G$, where we consider $k=\mathbb R$ for real Lie groups $G$ and  $k=\mathbb C$ for complex Lie groups $G$.  We say $\lambda$ is an algebraic 1-parameter subgroup if the map $\lambda$ is regular (i.e. a map between algebraic groups). 

\end{defin}

\begin{remark}  Let $\lambda$ be a 1-parameter subgroup of $G$.  The image $\lambda(\mathbb R)$ is a subgroup of $G$ and we will often abuse notation by denoting this subgroup simply by $\lambda$.  

\end{remark}

Our definition  of 1-parameter subgroup is somewhat restrictive as it  does not include one-parameter subgroups of nilpotent Lie groups, however this is a standard definition when studying reductive algebraic groups.  Note that  algebraic 1-parameter subgroups are fully reducible.  (This can be worked out by hand, but it is also a special case of a general result of Mostow on regular representations of algebraic, reductive groups \cite{Mostow:FullyReducibleSubgrpsOfAlgGrps}.)

\subsection{Geometric Invariant Theory}

\begin{thm}\label{thm: HMC - classical}[Hilbert-Mumford criterion]  Let $G$ be an algebraic, fully reducible subgroup of $GL(V)$.  
Suppose the stabilizer subgroup $G_p$ is finite.  Then $G\cdot p$ is closed if and only if $\lambda \cdot p$ is closed for all algebraic 1-parameter subgroups $\lambda$ of $G$.
\end{thm}

The theorem above was proven over $\mathbb C$ by Mumford and extended to real algebraic groups by Birkes \cite{Birkes:OrbitsofLinAlgGrps}.   Applying the criterion twice, we have the following immediate consequence on the inheritance of closed orbits.
\begin{cor}\label{cor: HMC finite stab -  inheritance of closed} Let $G$ be a  fully reducible, algebraic group such that  $G\cdot p$ is closed and $G_p$ is finite.  For any fully reducible,  algebraic subgroup $G'$ of $G$, we have that   $G'\cdot p$ is closed. \end{cor}
 
\begin{proof}
	The orbit $G\cdot p$ being closed implies $\lambda \cdot p$ closed for all algebraic 1-parameter subgroups $\lambda$ of $G$.  Now consider only those $\lambda$ which are subgroups of $G'$.  The stabilizer subgroup $G'_p \subset G_p$ is finite and so applying the Hilbert-Mumford criterion, we see that $G'\cdot p$ is closed.
\end{proof}
The statement of the corollary is rather strong and, in fact, does not hold  without the condition on the stabilizer; see, e.g., \cite[Example 6]{Jablo:GoodRepresentations}.  However, the corollary generalizes in a partial and useful way, as we will see.
In the results below, we do not specify whether the ground field is  $\mathbb R$ or $\mathbb C$ as knowing the result for one ground field implies that it holds for the other, cf \cite{BHC,RichSlow}.  The following are well-known.

\begin{prop}\label{prop: closed orbit of Gp vs G^0p}  Let $G$ be a fully reducible, algebraic group.  Denote the connected component of the identity (in the Hausdorff topology) by $G_0$.  Then $G\cdot p$ is closed if and only if $G_0\cdot p$ is closed.
\end{prop}

We say a group $G$ is \textit{pre-algebraic} if $G$ is the connected component of the identity of an algebraic group.   Obviously, if $G$ is pre-algebraic we have $G = (\overline G)_0$ where $\overline G$ is the Zariski closure of $G$.  We note that   a pre-algebraic group is fully reducible if and only if its Zariski closure is fully reducible; this follows from the fact that a (not necessarily connected) algebraic group is fully reducible if and only if its Lie algebra is so \cite{Mostow:FullyReducibleSubgrpsOfAlgGrps}.

\begin{prop}\label{prop: Gp closed implies stabilizer is reductive}  Let $G$ be an algebraic group.  The stabilizer subgroup $G_p$ of a point $p$ is an algebraic group.  Further, if $G$ is fully reducible and the orbit $G\cdot p$ is closed, then $G_p$ is fully reducible as well.
\end{prop}

\begin{prop}\cite[Corollary 3.1]{Luna:AdherencesDOrbiteEtInvariants}\label{prop: luna on normalizer of a stab grp having closed orbit} Let $G$ be a fully reducible, algebraic group.  Let $H$ be a fully reducible, algebraic subgroup of $G$ which stabilizes a point $p$.  Then  $G\cdot p$ is closed if and only if $N_G(H)\cdot p$ is closed. 
\end{prop}

From this result, we have the following useful lemma.

\begin{lemma}\label{lemma: luna on centralizer of a stab grp having closed orbit} Let $G$ be a fully reducible, pre-algebraic group and $H$ a connected, fully reducible (not necessarily algebraic) subgroup of $G$ which stabilizes a point $p$.  Then $G\cdot p$ is closed if and only if $Z_G(H)_0\cdot p$ is closed. 
\end{lemma}

\begin{proof}  We prove the lemma first in the special case that $G$ is algebraic.  Let $\overline H$ denote the Zariski closure of $H$.  First observe that  $Z_G(H) = Z_G(\overline H)$.  As $H$ is fully reducible, we have $H = [H,H] \ Z(H)$, where $Z(H)$ is the center of $H$.  Using the fact that connected, linear semisimple groups are necessarily pre-algebraic, we have $[H,H]$ is pre-algebraic and so $\overline H = \overline{[H,H]} \ \overline{Z(H)}$, where $(\overline{[H,H]})_0=[H,H]$.  

As $H$ is fully reducible we have that $\overline H$ is fully reducible.  Additionally, $\overline H$ is algebraic and so 
	$$N_\mathfrak g(\overline{\mathfrak h}) = Z_\mathfrak g (\overline{\mathfrak h}) + \overline{\mathfrak h},$$
which implies $N_G(\overline H)_0 =   Z_G(\overline H)_0 \cdot \overline H _0=   Z_G(H)_0 \cdot H$.  
As $H$ stabilizes $p$ we see that $N_G(\overline H)_0 \cdot p = Z_G(H)_0 \cdot p$.  Noting that $Z_G(H)$ is algebraic and applying Propositions \ref{prop: closed orbit of Gp vs G^0p} \& \ref{prop: luna on normalizer of a stab grp having closed orbit}, the lemma follows for $G$ algebraic.

For the general case, consider a pre-algebraic group $G_0$ with Zariski closure $G$.  By Proposition \ref{prop: closed orbit of Gp vs G^0p} and the work above, we have $G_0\cdot p$ is closed if and only if $Z_G(H)_0\cdot p$ is closed.  However, $Z_G(H)_0 = Z_{G_0}(H)_0$ as they have the same Lie algebra and the lemma is proven.
\end{proof}

\begin{cor} Let $G$ be a fully reducible, pre-algebraic group.  The orbit $G \cdot p \subset V$ is closed if and only if the stabilizer $H=(G_p)_0$ is fully reducible and $Z_G(H)_0\cdot p$ is closed. 
\end{cor}


The following is a slight generalization of the Hilbert-Mumford criterion.

\begin{lemma} Let $G$ be a fully reducible, (pre-)algebraic   group such that $H=(G_p)_0 \subset Z(G)$, the center of $G$.    
Then $G\cdot p$ is closed if and only if $\lambda \cdot p$ is closed for all (pre-)algebraic 1-parameter subgroups $\lambda$ of $G$. \end{lemma}
We note that the condition on $H$ makes it a fully reducible subgroup automatically.


\begin{proof}  Let $G$ be a pre-algebraic group with Zariski closure $\overline G$.  Observe that $Z(G)\subset Z(\overline G)$.  This fact together with Proposition \ref{prop: closed orbit of Gp vs G^0p} implies that it suffices to prove the lemma in the  case that $G$ and $\lambda$ are algebraic.  

We begin  by decomposing $G$ as a product $G = IH$, where $I$ is a fully reducible, algebraic subgroup with finite stabilizer.  The group $I$ is defined as the Lie group whose Lie algebra is defined by
	$$Lie~I = \{ X\in \mathfrak g \ | \ \mytrace(XY)=0 \mbox{ for all } Y\in \mathfrak h \}.$$
Details for showing $I$ is fully reducible and algebraic are the same as those given in  Section 6 of \cite{Jablo:ConceringExistenceOfEinstein}, see the discussion after Proposition 6.6.  To see that $G = IH$ is a product of groups, it suffices to show $\mathfrak g = \mathfrak i + \mathfrak h$ is a direct sum of Lie algebras.  

By hypothesis, $\mathfrak h$ commutes with all of $\mathfrak g$ and so commutes with $\mathfrak i$; hence, we simply need to show the sum $\mathfrak g = \mathfrak i + \mathfrak h$ is a vector space direct sum.  As $H$ and $G$ are pre-algebraic subgroups of some $GL(V)$, there exists some inner product on $V$ under which $H$ and $G$ are self-adjoint \cite{Mostow:SelfAdjointGroups}.  Using the resulting inner product on $\mathfrak{gl}(V)$, we see that $\mathfrak i$ is precisely the orthogonal complement of $\mathfrak h$ in $\mathfrak g$, thence we have the direct sum $\mathfrak g = \mathfrak i + \mathfrak h$.

Let $\lambda$ be a (pre-)algebraic 1-parameter subgroup of $G$.  Observe that $\lambda = \lambda_1   \lambda_2 $, where $\lambda_1 $ is a (pre-)algebraic 1-parameter subgroup of $I$ and $\lambda_2$ is a (pre-)algebraic 1-parameter subgroup of $H$.
As $H$ stabilizes $p$, we see from Theorem~\ref{thm: HMC - classical} that $G\cdot p = I\cdot p$ is closed if and only if $\lambda   \cdot p = \lambda_1 \cdot p$ is closed for all (pre-)algebraic 1-parameter subgroups $\lambda_1 $ of $I$ or, equivalently, for all (pre-)algebraic 1-parameter subgroups $\lambda$ of $G$.
\end{proof}

The next result on the inheritance of closed orbits is one of the main technical results needed in the proof of the Key Lemma (see Section \ref{sec: proof of key lemma}).

\begin{cor}\label{cor:inheritance of closed orbits}  Let $G$ be a fully reducible, pre-algebraic  group such that $(G_p)_0 \subset Z(G)$.  
Consider a fully reducible, pre-algebraic   subgroup $G'$ of $G$.  If $G\cdot p$ is closed, then so is the orbit $G'\cdot p$.
\end{cor}

 One proves this corollary in the same way that one proves Corollary \ref{cor: HMC finite stab -  inheritance of closed}.  We note that $(G'_p)_0$ being central in $G'$ follows from  $(G_p)_0$ being central in $G$.
In the sequel, we make use of the following well-known proposition.

\begin{prop}\label{prop: normal subgroups have closed orbits}  Let $G$ be a fully reducible, pre-algebraic subgroup of $GL(V)$ such that $G\cdot p$ is closed for some $p\in V$.  If $N$ is a normal, fully reducible, pre-algebraic subgroup of $G$, then $N\cdot p$ is closed.
\end{prop}

\begin{remark}
As this result is known to many working in geometric invariant theory, a reference is hard to find and so we provide a short argument for completeness.
\end{remark}

\begin{proof}  By Proposition \ref{prop: closed orbit of Gp vs G^0p}, we may assume that $G$ and $N$ are algebraic.

Since $N$ is a fully reducible, algebraic group acting on the closed variety $G\cdot p$, we know that there exists $g\in G$ such that $N\cdot gp$ is closed.  In fact, this is true for `almost all' $g\in G$; this is the main result of \cite{Luna:ClosedOrbitsofReductiveGroups}.  However, $N\cdot gp = g (N\cdot p)$, as $N$ is normal.  The map $g:V\to V$ being a homeomorphism gives that $N\cdot p$ is closed as well.
\end{proof}

\section{Proof of the Key Lemma~\ref{key}}\label{sec: proof of key lemma}

The goal of this section is to prove the Key Lemma~\ref{key}, which we restate here for convenience:

\begin{lemma}[Key Lemma]\label{k} Let $S$ be a solvable Lie group of Einstein type.   Suppose that $S$ can be written as a semi-direct product $S=S_1\ltimes S_2$ satisfying the following hypotheses:
\begin{itemize} 
\item $S_1$ isomorphically embeds as an Iwasawa subgroup in a semisimple Lie group $G_1$.    In particular, we can write $S_1=A_1N_1$, where $G_1=K_1A_1N_1$ is an Iwasawa decomposition.  
\item The adjoint action of $S_1$ on the ideal $\s_2$ of $\s$ extends to a representation of $G_1$ on $\s_2$.  We thus view $S$ as a subgroup of $G_1\ltimes S_2$.  
\end{itemize}
Then $S_2$ is of Einstein type.
\end{lemma}

\begin{hypoth}\label{first.simp}  We may assume that $G_1$ is semisimple of noncompact type.  Indeed, in the language of Lemma~\ref{k}, the group $S_1$ lies in the noncompact part $G_{nc}$ of $G_1$, and the hypotheses of the Lemma trivially remains true if we replace $G_1$ by $G_{nc}$.

\end{hypoth}

We will apply Nikolayevsky's technique, described in Subsection~\ref{subsec.nik} to prove Lemma~\ref{k}.   

\subsection{The pre-Einstein derivation of $N_2$.}\label{subsec.pre-einst}

\begin{nota}\label{nota.rho}  The representation of $G_1$ on $\s_2$ in Lemma~\ref{k} leaves the nilradical $\n_2$ of $\s_2$ invariant.   We will denote by $\rho:\g_1\to \operatorname{End}(\n_2)$ the induced representation of the Lie algebra $\g_1$ on $\n_2$.
\end{nota}

\begin{lemma}\label{lem.jab} Let $H$ be a semisimple Lie group of noncompact type and let $\rho: H\to GL(V)$ be a finite-dimensional representation.  Let $H=KS$ be an Iwasawa decomposition.  If an element $T\in End(V)$ commutes with all elements of $\rho(S)$, then it commutes with all of $\rho(H)$.  

\end{lemma}

\begin{proof} $H$ acts on $End(V)$ by $(h,T)\mapsto \rho(h)T\rho(h)^{-1}$ for $h\in H$, $T\in End(V)$.   If $T$ commutes with $\rho(S)$, then the orbit of $T$ under the action of $H$ is compact.  By \cite[Lemma 7.2]{Jablo:StronglySolvable}, every compact orbit of a finite-dimensional representation of a semisimple Lie group of noncompact type consists of a single point.   The lemma follows.

\end{proof}

\begin{lemma}\label{lem.ady} Let $H$ be a connected Lie group and $N$ its nilradical.   Let $W\in\h$, and suppose that $\ad(W)|_\n:\n\to\n$ is a non-singular derivation.   Then the orbit of $W$ under $\Ad_H(N)$ is given by
$$\Ad_H(N)(W)=\{W+X: X\in\n\}.$$

\end{lemma}

\begin{proof}  We induct on the step size of the nilpotent Lie algebra $\n$.   If $\n$ is abelian, then we have 
$\Ad(\exp(Y))(W)=W+[Y,W]=W-\ad(W)(Y)$.   Thus the lemma follows in this case from the non-singularity of $\ad(W)|_\n$.

For the general case, let $X\in\n$.  We need to find $n\in N$ such that $\Ad_H(n)(W)=W+X$.  Write $[N,N]$ for the normal subgroup of $H$ with Lie algebra $[\n,\n]$, and let $\bar{H}=H/[N,N]$.  Let $\pi:H\to \bar{H}$ be the homomorphic projection.   Since $\bar{H}$ has abelian nilradical $\bar{N}=N/[N,N]$, there exists $\bar{n}_1\in \bar{N}$ such that $\Ad_{\bar{H}}(\bar{n}_1)(\bar{W})=\bar{W}+\bar{X}$.   Choose $n_1\in N$ such that $\pi(n_1)=\bar{n}_1$.   We then have 
$$V:=\Ad(n_1)(W)\equiv W + X\,\,\,\rm{mod}\,[\n,\n].$$
 Set $U=(W+X)-V\in[\n,\n]$.   Note that $\ad(V)|_{\n}$ is a non-singular derivation and thus restricts to a non-singular derivation of $[\n,\n]$.   Let $\s$ be the subalgebra of $\h$ given by $\s=\R V+ [\n,\n]$ and let $S$ denote the corresponding connected subgroup of $H$.
 The Lie algebra $\s$ has nilradical $[\n,\n]$.   Since the step size of $[\n,\n]$ is less than that of $\n$, the inductive hypothesis gives us an element $n_2$ of $[N,N]$ such that $\Ad_S(n_2)(V)=V+U$.   Note that $\Ad_H(n_2)(V)=\Ad_S(n_2)(V).$   Let $n=n_2n_1$.   We then have 
 $$\Ad_H(n)(W) =\Ad_H(n_2)(V)=V+U=W+X$$
 and the lemma follows.
\end{proof}

\begin{prop}\label{prop.a2}  We assume the hypotheses of Lemma~\ref{k} and~\ref{first.simp}.  Then, letting $\n_2=\nilrad(\s_2)$, there exists an abelian complement $\af_2$ of $\n_2$ in $\s_2$ such that:
\begin{enumerate} 
\item Letting $\af=\af_1+\af_2$ and $\n=\n_1+\n_2$, then $\s=\af+\n$ is a standard decomposition of $\s$.  (See Remark~\ref{rem.maxred} for the definition of standard decomposition.)
\item $\af_2$ commutes with $\g_1$.
\item There exist an element $W=W_1+W_2\in \af$ with $W_i\in \af_i$, $i=1,2$ such that, in the notation of~\ref{nota.ss}, $\varphi:=\ad(W)|_{\n}^\R$  is a pre-Einstein derivation of $\n$ and $\varphi_2:=\ad(W_2)|_{\n_2}^\R$ is a pre-Einstein derivation of $\n_2$.  
\item $\varphi_2$ is positive-definite.
\end{enumerate}
\end{prop}

\begin{proof}[Proof of Proposition~\ref{prop.a2}]

 We view $\s_1$ as a subalgebra of $\g_1$.  By the second hypothesis, the representation $X\mapsto \ad(X)|_{\s_2}$ of $\s_1$ extends to a representation $\rho: \g_1\to End(\s_2)$.  Thus we may view $\s$ as a subalgebra of the semi-driect sum $\g_1\ltimes \s_2$.  Since $\rho(\s_1)$ is an Iwasawa subalgebra of $\rho(\g_1)$, we see that $\rho(\af_1)$ consists of fully reducible elements.  Since also $\af_1$ acts fully reducibly on $\n_1$, it follows that $\ad_\s(\af_1)$ consists of fully reducible elements.  Let $\af'$ be a maximal fully $\ad$-reducible abelian subalgebra of $\s$ containing $\af_1$, and let $\n=\nilrad(\s)$.  By Remark~\ref{rem.maxred}, $\s=\af'+\n$ is a standard decomposition.

Define $\af'_2 := (\mathfrak n_1 + \mathfrak s_2) \cap \mathfrak a'$.  Since $\s=\af_1+\n_1+\s_2$ (vector space direct sum) and $\af_1\subset \af'$, we have 
$\af'=\af_1+\af'_2$.   From the facts that  $\mathfrak a_1$ commutes with $\af ' $, normalizes each of $\n_1$ and $\s_2$, and contains an element which is non-singular on $\mathfrak n_1$, we see that $\af'_2 \subset \s_2$.   Thus $\s_2=\af'_2+\n_2$.  

 By Propositions~\ref{prop.heb} and \ref{prop.nik}, there exists an element $W'\in\af'$ such that $(\ad(W')|_{\n})^\R$ is a pre-Einstein derivation of $\n$.   Write $W'=W_1'+W_2'$ with $W_i'\in\af'_i$.   
Since $[\g_1,\s_2] \subset \n_2$, there exists a subspace $\cf$ of $\s_2$ complementary to $\n_2$ such that $[\g_1,\cf]=0$.  Write $W'_2=W_2+X$ with $W_2\in\cf$ and $X\in \n_2$.  Since both $\cf$ and $\af'_2$ commute with $\af_1$, so does the element $X$; i.e., $X$ lies in the zero-eigenspace $\q$ of $\ad(\af_1)|_{\n}$.  The eigenspace $\q$ is a subalgebra of $\n_2$.   Since $\af'_2$ commutes with $\af_1$, $\ad(\af'_2)$ leaves $\q$ invariant.  Moreover, since $[W',\n]=\n$ and $[W_1',\q]=0$, we must have $[W_2',\q]=\q$.   Thus $\ad(W_2')|_{\q}$ is a non-singular semisimple derivation of the nilpotent Lie algebra $\q$.  By Lemma~\ref{lem.ady} there exists an element $Y\in \q$ such that $\Ad(\exp(Y))(W_2')=W_2'-X=W_2$.   Set $W_1=W_1'$ and $W=W_1+W_2$.    Observe, $W = \Ad(\exp(Y))(W')$ is a pre-Einstein derivation.  
Let $\af_2=\Ad(\exp(Y))(\af'_2) \subset \s_2$ and set $\af=\af_1+\af_2 $. 

We need to show that $\af$ and $W$ satisfy conditions (i)--(iv).   Noting that $\af=\Ad(\exp(Y))(\af')$, we see that $\af+\n$ is again a standard decomposition of $\s$, and thus we have (i).

We next prove (iv).   Let $\lambda$ be an eigenvalue of $\ad(W_2)|_{\n_2}$ and $V_\lambda$ the corresponding eigenspace.   Then $V_\lambda$ is $\rho(\g_1)$-invariant since $W_2$ commutes with $\rho(\g_1)$.   Since $\g_1$ is semisimple, we must have 
$\mytrace(\ad(X)|_{V_\lambda})=0$ for all $X\in\g_1$ and thus 
	$$\mytrace(\ad(W)_{V_\lambda}=\mytrace(\ad(W_2)_{V_\lambda})=\lambda\,\dim(V_\lambda).$$  Since $\varphi=\ad(W)|_{\n}^\R$ is the pre-Einstein derivation of the Einstein nilradical $\n$, all eigenvalues of $\ad(W)_{\n_2}$ have positive real part by Proposition~\ref{prop.heb} and thus we must have $\Real(\lambda)>0$.   This proves (iv).

We next prove that $\af_2$ satisfies (ii).  Since $W_2\in\cf$, we have $[W_2,\g_1]=0$.   On the other hand, $[W_2,\n_2]=\n_2$ by (iv).  Hence $\g_1+\af_2$ is the zero-eigenspace of $\ad_{\g_1\ltimes\s_2}(W_2)$.   All elements of $ad_{\g_1\ltimes \s_2}(\af_2)$ commute with $\ad_{\g_1\ltimes\s_2}(W_2)$ and thus leave $\g_1+\af_2$ invariant.   Since $[\g_1,\s_2]\subset \n_2$, we thus have
$[\af_2,\g_1]\subset (\g_1+\af_2)\cap \n_2=(0)$, proving (ii).

(iii) We are left to prove that $\varphi_2$, as defined in (iii), is the pre-Einstein derivation of $\n_2$; i.e., that 
\begin{equation}\label{preEcond} \mytrace(\varphi_2 D)=\mytrace(D)\end{equation} for all $D\in \Der(\n_2)$.   

We use the shorthand notation $\Der$ for $\Der(\n_2)$.  As $\Der$ is the Lie algebra of the algebraic group $\Aut(\n_2)$, it has a  Levi decomposition 
	$$\Der = \Der_1 + \Der_2,$$ 
where $\Der_1$ is a maximal semisimple subalgebra and  the radical $\Der_2$ splits as a semi-direct sum 
	$$\Der_2=\Der^{\abc}+\nilrad(\Der)$$
with $\Der^{\abc}$ an abelian subalgebra commuting with $\Der_1$.   The subalgebra $\Der_1 + \Der^{\abc}$ is a maximal fully reducible subalgebra of $\Der$ (i.e., maximal among all subalgebras of $\Der$ that act fully reducibly on $\n_2$), and  $\Der^{\abc}$  consists of semisimple derivations of $\n_2$.     The elements of $\nilrad(\Der)$ are nilpotent derivations.  By Mostow \cite[Theorem 4.1]{Mostow:FullyReducibleSubgrpsOfAlgGrps}, the maximal fully reducible subalgebra of $\Der$ is unique up to conjugation.  Every pre-Einstein derivation $\psi$ of $\n_2$ lies in the center of a maximal fully reducible subalgebra (see proof of \cite[Theorem 1.1(a)]{Nikolayevsky:EinsteinSolvmanifoldsandPreEinsteinDerivation}) and hence is conjugate to an element of $\Der^{\abc}$.

To prove that $\varphi_2$ is the pre-Einstein derivation, we are required to show that Equation~\ref{preEcond} holds for all $D\in\Der$.  However, the next lemma (with $\varphi_2$ playing the role of $\sigma$) shows that it in fact suffices to verify Equation~\ref{preEcond} only for a select subset of derivations $D$.   The lemma is motivated by the proof of Theorem 1 in \cite{Nikolayevsky:EinsteinSolvmanifoldsandPreEinsteinDerivation}.

\begin{lemma}\label{lemma:reducing pre-Einstein eqn}  Let $\sigma$ be a semisimple derivation of $\n_2$ with real eigenvalues.  Let $\mathfrak r$ be any fixed choice of subalgebra of $\Der_1 + \Der^{ab}$ which contains $\sigma$ and  $\Der^{ab}$.  If $ \mytrace(\sigma D)=\mytrace(D) $ holds for all $D \in \mathfrak r$ (cf.\ Eqn \ref{preEcond}), then $\sigma$ is a pre-Einstein derivation of $\n_2$.

\end{lemma}

\begin{remark} Although the pre-Einstein derivation is contained in $\Der^{ab}$, we are careful to note that, a priori, our subalgebra $\mathfrak r$ must explicitly contain  both $\sigma$ and $\Der^{ab}$. 
\end{remark}

\begin{proof}  Let $\psi$ denote the unique pre-Einstein derivation contained in $\Der^{ab}$.  As such, we know that
	$$\mytrace (\psi D) = \mytrace D \quad \mbox{ for all } D\in Der.$$
By hypothesis, we have that $\psi, \sigma \in \mathfrak r$ and that $\sigma$ satisfies Eqn \ref{preEcond} for $D\in\mathfrak r$.  This yields
	$$\mytrace (\psi - \sigma)^2 = \mytrace (\psi (\psi - \sigma)) - \mytrace (\sigma (\psi-\sigma)) = 0.$$

Now recall that $\psi$ is diagonalizable over $\mathbb R$ as it is a pre-Einstein derivation.  Together with the fact that $\sigma$ is diagonalizable over $\mathbb R$ (being positive definite) and that $\psi$ and $\sigma$ commute, we see that $\psi - \sigma$ is diagonalizable over $\mathbb R$.  

Finally, $\mytrace (\psi - \sigma)^2 = 0$ implies $\psi = \varphi_2$, thence $\sigma$ is a pre-Einstein derivation and so Eqn \ref{preEcond} holds for all derivations.

\end{proof}

To use the lemma above, we carefully select a subalgebra of $\Der_1 + \Der^{ab}$ which satisfies the hypotheses.  Given any subalgebra $\bfrak$ of $\Der(\n_2)$, denote by $\Der(\n_2)^\bfrak$ the subalgebra of all derivations that commute with $\bfrak$. 

\begin{lemma}\label{lem.extend}   Let 
$\mathfrak{e}=\{D\in \Der(\n)^\af: D|_{\n_1}=0\}.$  Then:
\begin{enumerate}
\item $\mathfrak{e}$ is an ideal in $\Der(\n)^\af$.
\item $\dercntwo=\{D|_{\n_2}: D\in \mathfrak{e}\}$.  (Here we are identifying $\af_2$ with $\ad(\af_2)|_{\n_2}$.)   
\item $\dercntwo=\Der(\n_2)^{\mathfrak{a}_2 + \rho(\mathfrak{g}_{1})}.$
\item $\mathfrak{e}$ acts fully reducibly on $\n$ and $\dercntwo$ acts fully reducibly on $\n_2$.
\item Equation~\ref{preEcond} holds for every $D\in \dercntwo$. 

\end{enumerate}

\end{lemma}

\begin{proof}
(i) Since $\af_2\subset \af$ and $\n_1$ is the zero-eigenspace of $\af_2$ while $\n_2=[\af_2,\n_2]$, all elements of $\Der(\n)^\af$ leave each of $\n_1$ and $\n_2$ invariant.  Thus $\mathfrak{e}$ is an ideal in $\Der(\n)^\af$.  

(ii) is elementary and (iii) follows from Lemma~\ref{lem.jab}.

(iv)  By Proposition~\ref{prop.heb}, $\Der(\n)^\af$ acts fully reducibly on $\n$ and thus also on $\n_2$.  Since $\mathfrak{e}$ is an ideal in $\Der(\n)^\af$, it also acts fully reducibly.   (See \cite{Mostow:FullyReducibleSubgrpsOfAlgGrps}, p. 208.)  The second statement in (iv) follows from (ii).

(v)  Let $D\in\dercntwo$.  By (ii), $D$ extends to a derivation $\tilde{D}\in \Der(\n)$ satisfying $\tilde{D}|_{\n_1}=0.$  Since $\varphi$ is the pre-Einstein derivation of $\n$, we have  (writing $\varphi_1=(\ad(W_1)|_{\n})^\R$)
$$\mytrace(D) = \mytrace(\tilde{D}) = \mytrace(\varphi \tilde{D}) = \mytrace( \varphi_1 \tilde{D}) + \mytrace(\varphi_2 D)$$
$$= \mytrace(\varphi_1|_{\mathfrak n_2}D) + \mytrace(\varphi_2D)$$ 
Thus we need only show that $\mytrace(\varphi_1|_{\mathfrak{n}_2}D) =0$.   By (iii), each eigenspace $V$ of $D$ in $\n_2$ is preserved by $\rho(\g_{1})$ and we have $\mytrace(\rho(X)|_V)=0$ for all $X\in\g_{1}$,  since every finite-dimensional representation of a semisimple Lie algebra is traceless.  This holds in particular for $X=W_1$, and hence we have $\mytrace(\varphi_1|_V)=0$ on each such eigenspace.   Thus  $\mytrace(\varphi_1|_{\mathfrak n_2}D) =0$ as was to be shown.  This completes the proof of the lemma.

\end{proof}

As $\dercntwo$, $\mathfrak a_2$, and $\rho(\mathfrak g_1)$ all  act  fully reducibly on $\n_2$ and commute, the subalgebra $\dercntwo + \mathfrak a_2  + \rho(\mathfrak g_1)$ acts fully reducibly.  Thus, there is some maximal reductive subalgebra $\Der_1 + \Der^{ab}$  which contains them all.  Clearly, $\Der^{ab} \subset \dercntwo$.  Since $\varphi_2 \in \dercntwo$, we may apply the lemmas above to see that, in fact, $\varphi_2$ is a pre-Einstein derivation of $\n_2$.  This completes the proof of Proposition \ref{prop.a2}.

\end{proof}

\begin{cor}\label{cor.final} To prove Lemma~\ref{k}, it suffices to show that $N_2$ admits a nilsoliton metric.
\end{cor}

Corollary~\ref{cor.final} follows from Proposition~\ref{prop.a2} and Corollary~\ref{cor.nik}.

We will carry out the proof of the existence of a nilsoliton metric on $N_2$ in the next subsection.

\subsection{Existence of a nilsoliton metric on $N_2$.}\label{subsec.nilsol}

By Proposition~\ref{prop.a2}, we know that $\s_2$ can be written as $\af_2+\n_2$ where $\mathfrak a_2$ is abelian and $[\g_1,\af_2]=0$.   Moreover, there exists an element $W=W_1+W_2\in \af$ with $W_i\in \af_i$, $i=1,2$ such that $\varphi:=\ad(W)|_{\n}^\R$  is a pre-Einstein derivation of $\n$ and $\varphi_2:=\ad(W_2)|_{\n_2}^\R$ is a pre-Einstein derivation of $\n_2$. 

\begin{hypoth}\label{hyp} In addition to the hypotheses of Lemma~\ref{k} and our first simplification~\ref{first.simp}, we claim that it suffices to prove the existence of a nilsoliton metric on $N_2$ under the following additional hypotheses on $S$:

\begin{enumerate} 

	\item All eigenvalues of $\ad(W_2)|_{\n_2}$ are real;  equivalently, $\varphi_2:=\ad(W_2)|_{\n_2}$.

	\item $\af_2$ is one-dimensional; equivalently, $\af_2=\R W_2$.   
\end{enumerate}

\end{hypoth}

Indeed, let $\s'_2$ be the semi-direct sum of $\R \varphi_2\ltimes\n_2$.  By Remark~\ref{nota.ss} and Proposition~\ref{prop.a2}, $\varphi_2$ commutes with the action $\rho$ of $\g_1$ on $\n_2$.  In particular, $\s'=\s_1\ltimes \s'_2$ is a well-defined solvable Lie algebra and, by Corollary~\ref{cor.nik}, 
  the corresponding simply-connected Lie group $S'$ is of Einstein type.  We can again use Remark  \ref{nota.ss} to form the semi-direct product $G_{1}\ltimes S'_2$.   Since the nilradical $N_2$ of $S'_2$ coincides with that of $S_2$ and since $S'$ satisfies the hypotheses of Lemma~\ref{k}, the claim follows.

 \begin{notarem}\label{note.nik2} In the notation of Proposition~\ref{prop.a2}, we will identify $W_2$ with $\varphi_2$ and thus write $\varphi_2\in \af_2$.   We will similarly identify the pre-Einstein derivation $\varphi$ of $\s$ with $W=W_1+\varphi_2$.  (Note that $W_1=W_1^\R$ since $W_1\in\af_1$ and $\s_1=\af_1+\n_1$ is an Iwasawa algebra.)
 
 (ii) We use the identifications in~\ref{note.nik}.  Let $n=\dim(\n)$ and $n_i=\dim(\n_i)$.   We identify $\R^n$ with $\R^{n_1}\times\R^{n_2}$ and let 
 $$i:GL(\R^{n_2}) \to GL(\R^n)$$
 be the associated embedding.  Thus 
 $$i(\alpha)=\begin{bmatrix}I&0\\0&\alpha\end{bmatrix}$$
 and $$i_*(X)=\begin{bmatrix}0&0\\0&X\end{bmatrix}$$
 for $\alpha\in GL(n_2,\R)$ and $X\in \gl(n_2,\R)$.
 In the notation of Proposition~\ref{prop.nik2}, we write $\mu$ for the Lie bracket of $\n=\n_1+\n_2$ and $\mu_i$ for the Lie bracket of $\n_i$ and view $\n$, respectively $\n_i$, as the vector space $\R^n$, respectively $\R^{n_i}$,  with bracket $\mu$, respectively $\mu_i$.  Under the identification of $\R^n$ with $\R^{n_1}\times \R^{n_2}$, we may write 
  \begin{equation}\label{eq.mu}\mu = \mu_1 + \mu_{12} + \mu_2,\end{equation}
where $\mu_{12}$ denotes the adjoint  action of $\mathfrak n_1$ on the ideal $\mathfrak n_2$.  Note that
$$\mu_{12}(X,Y)=\rho(X)Y \mbox{\,for all\,}X\in \n_1,\,Y\in\n_2$$
where $\rho:\g_1\to Der(\s_2)$ is the differential of the representation of $G_1$ in Lemma~\ref{k}.
   \end{notarem} 
   
 Let $G_\varphi<SL(n,\R)$ and $G_{\varphi_2}<SL(n_2,\R)$ be the pre-algebraic groups defined in Proposition~\ref{prop.nik2}.  To prove the Lemma~\ref{k}, we need to show that $ G_{\varphi_2}\cdot\mu_2$ is closed.   We will use Equation~\ref{eq.mu} and the fact that $G_\varphi\cdot\mu$ is closed.
 \begin{lemma}\label{lem.rho}\text{}  
 \begin{enumerate}
 \item $\rho(\g_1)<\g_{\varphi_2}$.  
 \item Let $\cf_\rho=Z_{\g_{\varphi_2}}(\rho(\g_1))$ be the centralizer of $\rho(\g_1)$ in $\g_{\varphi_2}$ and let $C_\rho$ be the corresponding connected subgroup of $G_{\varphi_2}$.  Then $i(C_\rho)<G_{\varphi}.$
 \item For $\alpha\in C_\rho$, we have 
 $$i(\alpha)\cdot\mu=\mu_1+\mu_{12}+\alpha\cdot\mu_{22}.$$
 
 \end{enumerate}
 \end{lemma}

 \begin{proof} (i) is immediate since $\rho(\g_1)$ consists of derivations of trace zero that commute with $\varphi_2$, as noted immediately following the statement of~\ref{hyp}.   For (ii), let $X\in\cf_\rho$.  Recalling that $\varphi=W_1+\varphi_2$ with $W_1\in\af_1$, we see that 
 $$[\varphi,i_*X]_{\gl(n,\R)}=i_*([\rho(W_1),X]_{\gl(n_2,\R)}+ [\varphi_2,X]_{\gl(n_2,\R)})=0.$$
 (Here $[\cdot\,\cdot]_{\gl(m,\R)}$ denotes the Lie bracket of $\gl(m,\R)$.)  Thus $i_*(X)\in \mathfrak z(\varphi)$.   Moreover $$\mytrace(\varphi \,i_*(X))=\mytrace(\varphi_2 X)=0.$$ 
 (The first equality follows from the fact that each element of $\rho(\g_1)$--in particular, $\rho(W_1)$-- restricts to a traceless representation on each eigenspace of $\varphi_2$, and the second equality is immediate from the definition of $\g_{\varphi_2}$ in Proposition~\ref{prop.nik2}.)  Hence $i_*X\in \g_\varphi$.  Finally, (iii) follows from Equation~\ref{eq.mu} and the fact that $\alpha$ commutes with $\rho(\g_1)$.
  
  \end{proof}

As noted in~\ref{note.nik}(i), the stabilizer of $\mu_2$ in $G_{\varphi_2}$ has Lie algebra $\Der(\n_2)\cap \mathfrak{sl}(n_2,\R)$.   Let $H$ be any connected fully reducible subgroup of the stabilizer and let $C:=Z_{G_{\varphi_2}}(H)_0$.  By Lemma~\ref{lemma: luna on centralizer of a stab grp having closed orbit}, $G_{\varphi_2}\cdot\mu_2$ is closed if and only if $C\cdot\mu_2$ is closed.   If, moreover, we choose $H$ so that its Lie algebra contains $\rho(\g_1)$, then by Lemma~\ref{lem.rho}, the latter condition is equivalent to $i(C)\cdot\mu$ being closed.   In the following corollary, we make a choice of $H$.

  \begin{cor}\label{cor.plan}   Let $\mathfrak{d}=\dercntwo\cap\mathfrak{sl}(n_2,\R)$ and let $\h=\rho(\g_1)+\mathfrak{d}<\Der(\n_2)$.  Let $\mathfrak{c}=Z_{\g_{\varphi_2}}(\h)$,  and let $C$ be the corresponding connected subgroup of $G_{\varphi_2}$.  Then the following are equivalent:
\begin{itemize}
 \item $N_2$ admits a nilsoliton metric; 
 \item the orbit $C\cdot\mu_2$ is closed;
 \item the orbit $i(C)\cdot\mu$ is closed.
\end{itemize}
  \end{cor}

  \begin{proof} By Lemma~\ref{lem.extend}, $\mathfrak{d}$ is fully reducible and it commutes with the fully reducible algebra $\rho(\mathfrak g_1)$; thus $\h$ is fully reducible.   Hence the equivalence of the first two conditions follows from Lemma~\ref{lemma: luna on centralizer of a stab grp having closed orbit} and Proposition~\ref{prop.nik2}.  The equivalence of the second and third conditions follows from Lemma~\ref{lem.rho}(iii) since $C<C_\rho$.
  
  \end{proof}

To complete the proof of the Key Lemma, we need to show that $i(C)\cdot\mu$ is closed.  We exploit the notion of `inheritance of closed orbits', see Section \ref{sec: technical lemmas on GIT}.

\begin{lemma}\label{thm:refinement of Niko}  Let $\g_\varphi^{\mathfrak a}$ be the subalgebra of $\g_\varphi$ consisting of all elements that commute with $\ad_\s(\af)|_\n$, and let $G_\varphi ^\mathfrak a$ be the corresponding connected subgroup of $G_\varphi$.   Then $G_\varphi ^\mathfrak a$ is a pre-algebraic fully reducible subgroup of $G_\varphi$, and $G_\varphi ^\mathfrak a \cdot \mu$ is closed.
\end{lemma}

\begin{proof}   By~\ref{note.nik}, $G_\varphi$ is fully reducible  and pre-algebraic.  By Definition~\ref{def.pre-Einst} and Proposition~\ref{prop.nik2}, $\ad(\af)\cap\ker(t)=\ad(\af)\cap \mathfrak{sl}(n,\R)$ and hence there exists a codimension one subspace $\af_0$ of $\af$ such that $\ad(\af_0)\subset \g_\varphi$ and $\af=\af_0+\R \varphi$.  Since all elements of $G_\varphi$ commute with $\varphi$, we have $G_\varphi^{\mathfrak a}=Z_{G_{\varphi}}(\ad(\af_0))$.   The group $\Ad(A_0)$ is an abelian group of real semisimple transformations and hence is fully reducible.  Hence $G_\varphi ^\mathfrak a$ is also pre-algebraic and fully reducible.  By Luna's result Lemma~\ref{lemma: luna on centralizer of a stab grp having closed orbit}, the orbit $G_\varphi ^\mathfrak a \cdot \mu$ is closed.

\end{proof}

\begin{lemma}\label{prop.h}  Let $F=\{X\in G_\varphi^\mathfrak a: X|_{\n_1}=Id\}_0$.   Then
	\begin{enumerate}
	\item $F$ is a fully reducible, pre-algebraic, normal subgroup of $G_\varphi ^\mathfrak a $.
	\item $F\cdot \mu$ is closed.
	\item The Lie algebra of the stabilizer $F_\mu$ coincides with $\mathfrak{e}\cap \mathfrak{sl}(n,\R)$ where $\mathfrak{e}$ is the subalgebra of $\Der(\n)$ defined in Lemma~\ref{lem.extend}.
		\item  Writing $E=(F_\mu)_0$, then $(Z_F(E))_0\cdot \mu$ is closed.
		\item $i(C)<(Z_F(E))_0$.		
	\end{enumerate}
			
\end{lemma}

\begin{proof}  Since $G_\varphi^{\mathfrak a}$ commutes with $\varphi_2$, it normalizes each of $\n_1$ and $\n_2$.    Statement (i) is thus immediate.   Statement (ii) follows from the closedness of the orbit $G_\phi ^\mathfrak{a} \cdot \mu $  together with Proposition \ref{prop: normal subgroups have closed orbits}, and Statement (iii) follows from Lemma~\ref{lem.extend}(i).   (iv) is a consequence of Lemma~\ref{lemma: luna on centralizer of a stab grp having closed orbit}.   Finally, (v) follows from the definition of $C$ in Corollary~\ref{cor.plan}; in fact $i(C)=(Z_F(E))_0\cap i(C_\rho)$.
\end{proof}

To complete the proof of the Key Lemma, note that $i(C)$ is a fully reducible, pre-algebraic group.  We apply Corollary~\ref{cor:inheritance of closed orbits} with $(Z_F(E))_0$ playing the role of $G$ and $i(C)$ playing the role of $G'$ to conclude that $i(C)\cdot\mu$ is closed.   The Key Lemma now follows from Corollary~\ref{cor.plan}.

%
%
%
%
%
%
%
%
%
%

\section{Extensions of soliton metrics}\label{exts}
A long standing and important question is to understand when a solvable or nilpotent Lie group admits an Einstein or Ricci soliton metric.  The first characterization on the existence of Ricci soliton metrics on nilpotent Lie groups was due to Lauret \cite{LauretNilsoliton}.

\begin{thm}[Lauret] A nilpotent Lie group $N$ admits a Ricci soliton metric if there exists an abelian group $A$ acting reducibly on $\mathfrak n$ such that $A\ltimes N$ admits an Einstein metric.
\end{thm}

This was later extended to solvable Lie groups for which the Ricci soliton is a so-called `algebraic Ricci soliton' \cite{Lauret:SolSolitons} and in \cite{Jablo:HomogeneousRicciSolitons} it was shown that all Ricci solitons on solvable Lie groups are algebraic.  Thus we have the following.

\begin{thm}[Lauret, Jablonski]\label{lj}  A solvable Lie group $S$ admits a Ricci soliton metric if there exists an abelian group $A$ acting reducibly on $\mathfrak s$ such that $A\ltimes S$ admits an Einstein metric.
\end{thm}

Upon inspection of the structure of these groups, one can drop the condition that $A$ act reducibly, a priori, as it can be deduced.

All known examples of non-compact, homogeneous Einstein and Ricci soliton spaces are isometric to solvable Lie groups with left-invariant Riemannian metrics.   Naturally, one asks if results analogous to Theorem~\ref{lj} are possible if $A$ is replaced with some non-abelian solvable Lie group.  The Key Lemma above (Lemma \ref{key}) is one such extension where a nilpotent Lie group is extended by a solvable group with some conditions on the action of the extension.

\begin{question}  Consider an extension $S_1 \ltimes N_2$ as in the Key Lemma.  Can we drop the hypothesis that the adjoint representation of $S_1$ on $\mathfrak n_2$ extend to a representation of the full semisimple to which $S_1$ belongs?
\end{question}

\begin{prop} There exists a  solvable group  $S_1$ and a nilpotent group $N_2$ such that $S = S_1\ltimes N_2$ admits an Einstein metric, but $N_2$ does not admit a soliton. 
\end{prop}

To build such an example, we begin by defining the bracket relations for a nilpotent Lie algebra which will be the nilradical $\mathfrak n$ of $\mathfrak s$.  Let $\mathfrak n = span\{ e_0, \dots , e_8,z_1,z_2 \}$ as a vector space.  We define a 2-step nilpotent Lie bracket structure on $\mathfrak n$ according to the following relations.
	$$\begin{array}{ll}
	\left[e_1,e_2\right] = \sqrt 8 \  z_1   \hspace{1cm} { }     &     \left[e_0,e_1\right] = \sqrt 8 \  z_2 \\
	\left[e_3,e_4\right] = \sqrt 12 \ z_1 &   \\
	\left[e_5,e_6\right] = \sqrt 3 \ z_1  & \left[e_5,e_8\right] = 3 \ z_2   \\
	\left[e_7,e_8\right] = \sqrt 3 \ z_1  & \left[e_6,e_7\right] = 3 \ z_2   
	\end{array}$$
Using skew-symmetry and bilinearity, these relations completely determine the Lie bracket.  Notice that the center of $\mathfrak n$ is spanned by $\{z_1,z_2\}$.

We now extend the Lie bracket on $\mathfrak n$ to a solvable Lie algebra of dimension one greater.  Let $\mathfrak s = \mathfrak a \ltimes \mathfrak n$ where $\mathfrak a$ is 1-dimensional and spanned by an element $A$ acting on $\mathfrak n$ by 
	$$ad~A|_\mathfrak n = diag\{21,17,21,19,19,19,19,19,19,38,38\},$$
where the diagonal matrix is relative to the basis $\{ e_0, \dots , e_8,z_1,z_2 \}$ of $\mathfrak n$.

If we choose $\{A,e_0,\dots,e_8,z_1,z_2\}$ to be an orthonormal basis of $\mathfrak s$, then the Lie group $S$ with Lie algebra $\mathfrak s$ has a left-invariant Einstein metric.

Now observe that $\mathfrak s_1 = span\{A,e_0\}$ is isomorphic to the Iwasawa algebra of the semisimple Lie algebra $\mathfrak{sl}(2,\mathbb R)$.  Consider $\mathfrak n_2 = span\{e_1,\dots,e_8,z_1,z_2\}$ and observe that $\mathfrak n_2$ is an ideal.  In this way, we have decomposed the group $S$ into a semi-direct product
	$$S=  S_1\ltimes N_2.$$
All that remains is to prove  $N_2$ does not admit a left-invariant Ricci soliton.  By Proposition 4.4 of \cite{Jablo:ModuliOfEinsteinAndNoneinstein}, we see that $N_2$ does not admit such a metric.  (In the notation of that work, we have $k=2$ and $n=1$.)

\providecommand{\bysame}{\leavevmode\hbox to3em{\hrulefill}\thinspace}
\providecommand{\MR}{\relax\ifhmode\unskip\space\fi MR }
\providecommand{\MRhref}[2]{%
  \href{http://www.ams.org/mathscinet-getitem?mr=#1}{#2}
}
\providecommand{\href}[2]{#2}

\end{document}